\documentclass[a4paper]{article}
\usepackage{latexsym,amsthm,amsmath,amssymb}

\newtheorem{theorem}{Theorem}
\newtheorem{lemma}{Lemma}

\usepackage{authblk}
\usepackage[utf8x]{inputenc}
\usepackage[english]{babel}
\usepackage{graphicx}
\usepackage{subcaption}
\usepackage{pbox}
\usepackage{type1ec}
\usepackage[T1]{fontenc}
\usepackage{bbm}
\usepackage{mathtools}
\usepackage{amssymb}
\usepackage{amsmath}
\usepackage{amsthm}
\usepackage{amsfonts}
\usepackage{pifont}
\usepackage{url}
\usepackage{multirow}
\usepackage{multicol}
\usepackage{fancybox}
\usepackage{colortbl}
\usepackage{soul}
\usepackage{color}
\usepackage{tikz}
\usepackage{tkz-berge}
\usepackage{rotating}
\usepackage{pdflscape}
\usepackage{afterpage}
\usepackage{comment}
\usepackage{etoolbox}
\usepackage{environ}
\usepackage{dirtytalk}

\usepackage{algorithm}
\usepackage{algorithmic}

\usetikzlibrary{decorations,arrows,shapes}

\usepackage{complexity}

\newtheorem{observation}{Observation}
\newtheorem{corollary}{Corollary}

\theoremstyle{definition}
\newtheorem{definition}{Definition}

\newcommand{\str}{\mathcal{S}}

\newcommand{\bigO}[1]{\mathcal{O}\left(#1\right)}

\newcommand{\tsc}[1]{\textsc{#1}}

\newcommand{\problem}[3]{{\centering\fbox{\pbox{\textwidth}{\tsc{#1}\\\textit{Instance}: #2\\\textit{Question}: #3}}}}

\newcommand{\inners}{1.2pt}
\newcommand{\outers}{1pt}
\newcommand{\gscale}{0.6}

\newcommand{\tdef}[1]{\textit{#1}}

\newclass{\Hard}{Hard}
\newclass{\Hness}{Hardness}

\newclass{\Complete}{Complete}
\newclass{\Cness}{Completeness}
\newcommand{\NPc}{\NP\text{-}\Complete}

\newfunc{\dist}{dist}
\newfunc{\girth}{girth}
\newfunc{\nd}{nd}
\newfunc{\YES}{YES}
\newfunc{\NOi}{NO}
\newfunc{\projax}{Proj}
\newfunc{\liftax}{Lift}
\newfunc{\ff}{ff}
\newfunc{\cf}{cf}

\renewcommand{\K}[1]{K_{\mathcal{#1}}}

\newtoggle{proof_toggle}
\toggletrue{proof_toggle}

\NewEnviron{tproof} {
    \iftoggle{proof_toggle} {%
        \begin{proof}\BODY\end{proof}}
    {%
    }
}

\newcommand{\ceil}[1]{\left\lceil#1\right\rceil}
\newcommand{\floor}[1]{\left\lfloor#1\right\rfloor}

\author[1]{Guilherme C. M. Gomes}
\author[2]{Marina Groshaus}
\author[1]{Carlos V. G. C. Lima}
\author[1]{Vinicius F. dos Santos}

\affil[1]{DCC, Universidade Federal de Minas Gerais, Belo Horizonte, Brazil}
\affil[2]{DAINF, Universidade Tecnol\'ogica Federal do Paran\'a, Curitiba, Brazil}

\date{}
\title{Intersection graph of maximal stars}

\begin{document}

\maketitle


\begin{abstract}
    A biclique of a graph $G$ is an induced complete bipartite subgraph of $G$ such that neither part is empty.
    A star is a biclique of $G$ such that one part has exactly one vertex.
    The star graph of $G$ is the intersection graph of the maximal stars of $G$.
    A graph $H$ is star-critical if its star graph is different from the star graph of any of its proper induced subgraphs.
    We begin by presenting a bound on the size of star-critical pre-images by a quadratic function on the number of vertices of the star graph, then proceed to describe a Krausz-type characterization for this graph class; we combine these results to show membership of the recognition problem in \textsf{NP}.
    We also present some properties of star graphs. In particular, we show that they are biconnected, that every edge belongs to at least one triangle, characterize the structures the pre-image must have in order to generate degree two vertices, and bound the diameter of the star graph with respect to the diameter of its pre-image.
Finally, we prove a monotonicity theorem, which we apply to list every star graph on at most eight vertices.
\end{abstract}


\section{Introduction}

Intersection graphs form the basis for much of the theory on graph classes.
A chordal graph, for instance, is intersection graphs of all subtrees of some tree~\cite{classes_survey}.
Interval graphs, in turn, are defined as the family of intersection graphs of subpaths of a path.
Line graphs, are the intersection graphs of the edges of some graph.
Unlike chordal graphs, there are known characterizations for line graphs that make use of a finite family of forbidden induced subgraphs~\cite{line_nich}.
Moreover, line graphs were one of the first classes to be characterized in terms of edge clique covers that satisfy some properties pertinent to the intersection definition; results of this form are known as \textit{Krausz-type characterizations}.

All of the aforementioned classes are easily recognizable in polynomial time~\cite{classes_survey,line_naor}.
The complexity of recognising clique graphs -- the intersection graphs of the maximal cliques of some graph -- on the other hand, was left open for several years, with a very complicated argument, due to Alcón et al.~\cite{clique_recognition}, showing that the problem is $\NPc$.
Aside from the complexity point of view, many different properties of clique graphs have been investigated in the literature.
For instance, clique-critical graphs -- graphs whose clique graph is different from the clique graph of all of its proper induced subgraphs -- have different characterizations~\cite{clique_critical_toft} and bounds~\cite{clique_critical_alcon} which were crucial in the proof of the complexity of the recognition problem.
Another common line of investigation on intersection graphs is the behaviour of iterated applications of the operators.
For instance, Frías-Armenta et al.~\cite{clique_iterated}, and Larrión and Neumann-Lara~\cite{clique_divergent} study iterated applications of the clique operator.
Biclique graphs -- the intersection graph of the maximal induced complete bipartite graphs of a graph -- were first characterized and studied by Groshaus and Szwarcfiter~\cite{biclique_graph}.
Their results, however, are not very useful from the algorithmic point of view, and do not appear yield many insights on the recognition problem.
Nevertheless, they study the behaviour of biclique graphs, showing that every edge is contained in a diamond or in a 3-fan, and specialize their general characterization for biclique graphs of bipartite graphs.
As was done for clique graphs, the iterated biclique operator has also been studied by Groshaus et al. in multiple papers~\cite{biclique_iterated, almost_all_biclique}, with results ranging from characterizations of divergence, divergence type verification algorithms, and other structural results.

Regarding stars, previous work handled the intersection graphs of (not necessarily maximal) substars of a tree~\cite{substar_graph} and of a star~\cite{starlike_graph}.
For the first, a minimal infinite family of forbidden induced subgraphs was given, while, for the latter, a series of characterizations were shown, including a finite family of forbidden induced subgraphs.
Stars are a particular case of bicliques, and both the biclique graph and star graph coincide for $C_4$-free graphs.
In fact, this relationship was successfully applied to determine the complexity of biclique coloring (a coloring of the vertices of a graph such that no induced maximal biclique is monochromatic)~\cite{biclique_coloring_complexity}, using a reduction from \textsc{qsat}$_2$ to star coloring (a coloring of the vertices of a graph such that no induced maximal star is monochromatic).
To the best of our knowledge, these are the main topics discussed in the literature that involve maximal stars in some way.
However, star graphs appear to be natural generalizations of square graphs~\cite{murty} in the sense that
the squaring operation essentially picks the non-induced star centered at each vertex, and the intersection graph of these stars is generated.
On the other hand, for star graphs, every \textit{induced} maximal star is used in the construction of the intersection graph.
Despite the classes of star graphs and biclique graphs being equivalent when restricting the pre-image domain to $C_4$-free graphs, we were unable to deepen the study of biclique graphs; our efforts were hindered by some of the questions posed and developed upon in this work.

In this paper, we are concerned exclusively with induced maximal stars, as such, when mentioning stars, we will always mean induced stars, unless otherwise noted.
We begin by characterizing the class of intersection graphs of maximal stars of some graph, which we call star graphs and then, using a bound on the maximum number of vertices that a minimal pre-image may possess, we prove membership of the star graph recognition problem in \textsf{NP}.
We then proceed to show some properties of members of the class.
Specifically, we prove that all of them are biconnected, that every edge belongs to at least one triangle, characterize the necessary pre-image structures for the existence of degree two vertices, and a tight bound on the diameter of the obtained star graph.
Afterwards, we present two small graphs that show that square graphs and star graphs are not subclasses of each other; we also list all connected star graphs on four, five, and six vertices.
To achieve this, we present a monotonicity theorem which implies that a minimal pre-image $H$ has more induced maximal stars than any induced subgraph of $H$.
We conclude this work with some open problems and future research directions.
\section{Preliminaries}
\label{sec:prelim}

We use standard graph theory notation~\cite{murty}.
A graph is a \tdef{star} on $p+1$ vertices if it is a tree and has a vertex of degree $p$.
We denote this graph by $K_{1,p}$, and say that an induced subgraph $G'$ of $G$ is an \textit{induced star} if $G'$ is isomorphic to $K_{1,p}$, with $p = |V(G')| - 1 \geq 1$.
An induced star is \tdef{maximal} if there is no superset of its vertices that is also an induced star.
The \tdef{center} of a star $s$, denoted by $c(s)$, is the vertex of maximum degree $s$, all other vertices are its \tdef{leaves}; if the star is a single edge, we define the center arbitrarily.
For two vertices $u,v \in V(G)$, let $\dist_G(u,v)$ denote the length of the shortest path between $u$ and $v$ in $G$.
The \tdef{$k$-th power} $G^k$ of a graph $G$ is defined with $V(G^k) = V(G)$, $E(G^k) = \{uv \mid \dist_G(u,v) \leq k\}$.
When $k = 2$, we refer to $G$ as the \tdef{square} of $H$.
The \tdef{intersection graph} of a multifamily $\mathcal{F} \subseteq 2^S$, denoted by $G = \Omega(\mathcal{F})$ is the graph of order $|\mathcal{F}|$ and, for every $F_u, F_v \in \mathcal{F}$, $uv \in E(G) \Leftrightarrow F_u \cap F_v \neq \emptyset$.
An \tdef{edge clique cover} $\mathcal{Q} = \{Q_1, \dots, Q_n\}$ of a graph $G$ is a (multi)family of cliques of $G$ such that every edge of $G$ is contained in at least one element of $\mathcal{Q}$.

\begin{definition}\label{def:star_graph}
    Let $\mathcal{G}$ be the set of all finite graphs and $\str(H)$ be the set of all induced maximal stars of a graph $H$. The \tdef{star operator} is the function $\K{S} : \mathcal{G} \mapsto \mathcal{G}$ such that, $\K{S}(H) = \Omega(\str(H))$.
    If $G = \K{S}(H)$, we say that $H$ is a \tdef{pre-image} of $G$ and that $G$ is the \tdef{star graph} of $H$.
    The \textit{iterated star operator} $\K{S}^i$ is defined as $\K{S}^1(H) = \K{S}(H)$ and $\K{S}^i(H) = \K{S}(\K{S}^{i-1}(H))$.
\end{definition}

When detailing which vertices belong to a star, we shall describe it by $\{v_1\}\{v_2, \dots, v_{p+1}\}$, with $v_1$ being its center and the other $p$ vertices its leaves.
Unless noted, $G$ will be our star graph and $H$ the pre-image of $G$.%
By Definition~\ref{def:star_graph}, two stars $s_a,s_b$ intersect if they share at least one vertex, with the possible cases being: (i) the centers of $s_a$ and $s_b$ coincide; (ii) the center of $s_a$ is a leaf of $s_b$; or (iii) $s_a$ and $s_b$ share at least one leaf.
Note that conditions (i) and (iii) may be simultaneously satisfied.
For an example of the intersection possibilities, please refer to Figure~\ref{fig:intersection_example}.

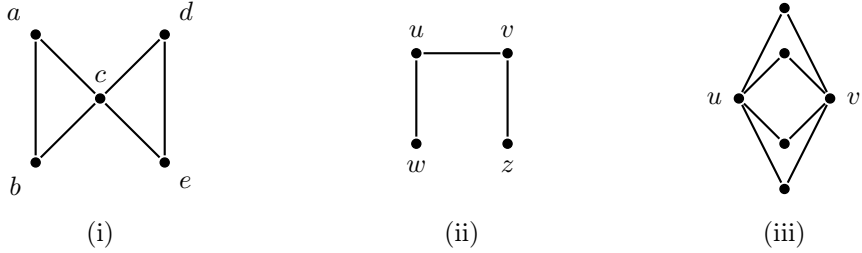
\begin{figure}[!tb]
    \centering
        \begin{tikzpicture}[scale=1.2]
            \GraphInit[unit=3,vstyle=Normal]
            \SetVertexNormal[Shape=circle, FillColor = black, MinSize=3pt]
            \tikzset{VertexStyle/.append style = {inner sep = \inners,outer sep = \outers}}
            \SetVertexLabelOut
            \begin{scope}[xshift=-3.5cm]
                \Vertex[x=0,y=0, Math, Lpos=90]{c}
                
                \Vertex[a=135, d=1, Math, Lpos=135]{a}
                \Vertex[a=225, d=1, Math, Lpos=225]{b}
                
                \Vertex[a=315, d=1, Math, Lpos=315]{e}
                \Vertex[a=45, d=1, Math, Lpos=45]{d}
                \node at (0,-1.5) {(i)};
                \Edges(b,c,d,e,c,a,b)
            \end{scope}
            \begin{scope}[xshift=-1, yshift=0.5cm]
                \Vertex[x=0,y=0, Math, Lpos=90, L={u}]{up}
                \Vertex[x=1,y=0, Math, Lpos=90, L={v}]{vp}
                
                \Vertex[x=0,y=-1, Math, Lpos=270]{w}
                \Vertex[x=1,y=-1, Math, Lpos=270]{z}
                \node at (0.5,-2) {(ii)};
                \Edges(w,up,vp,z)
            \end{scope}
            \begin{scope}[xshift=3.5cm]
                \Vertex[x=0,y=0, Math, Lpos=180]{u}
                \Vertex[x=1,y=0, Math, Lpos=0]{v}
                
                \Vertex[x=0.5,y=0.5, Math, NoLabel]{x_1}
                \Vertex[x=0.5,y=1, Math, NoLabel]{x_2}
                \Vertex[x=0.5,y=-0.5, Math, NoLabel]{x_3}
                \Vertex[x=0.5,y=-1, Math, NoLabel]{x_4}
                
                \node at (0.5,-1.5) {(iii)};
                \Edges(u,x_1,v,x_2,u,x_3,v,x_4,u)
            \end{scope}
        \end{tikzpicture}
        \caption{(i) The stars $\{c\}\{a,e\}$ and $\{c\}\{b,d\}$ intersect only at their center; (ii) the center of $\{u\}\{w,v\}$ is a leaf of star $\{v\}\{u,z\}$; (iii) the star centered at $u$ intersects the star centered at $v$ only at their leaves.\label{fig:intersection_example}}
\end{figure}

We say that star $s_a$ \textit{absorbs} star $s_b$ if, by removing one leaf of $s_b$, it becomes a substar of $s_a$.
A vertex $v$ is said to be \tdef{star-critical} if its removal changes the resulting star graph; that is, the star graph of $H$ and the star graph of $H \setminus \{v\}$ are not isomorphic.
Similarly to clique-critical graphs~\cite{clique_critical_toft,clique_critical_alcon}, a graph is \textit{star-critical} if all of its vertices are star-critical.
All vertices of Figure~\ref{fig:example_critical} are star-critical; in particular, the removal of $x$ does not cause the absorption of any star, but the intersection of two maximal induced stars is precisely $x$, i.e., there is an edge of the star graph that depends on $x$ to exist.
It is not hard to see that the only vertices which may be non-star-critical are simplicial vertices; for example, if there is a class of non-adjacent simplicial vertices that have the same neighborhood, all but one of them are certainly non-star-critical.
For the entirety of this work, we assume that all of our graphs are connected.

\begin{figure}[!tb]
    \centering
        \begin{tikzpicture}[scale=1.2]
            \GraphInit[unit=3,vstyle=Normal]
            \SetVertexNormal[Shape=circle, FillColor = black, MinSize=3pt]
            \tikzset{VertexStyle/.append style = {inner sep = \inners,outer sep = \outers}}
            \SetVertexLabelOut
            \begin{scope}
                \Vertex[x=-0.5,y=0, Math, Lpos=135]{a}
                \Vertex[x=-0.5,y=-1, Math, Lpos=225]{b}
                \Vertex[x=0.5,y=-1, Math, Lpos=315]{c}
                \Vertex[x=0.5,y=0, Math, Lpos=45]{d}
                \Vertex[x=0,y=0.5, Math, Lpos=90]{x}
                \Edges(a,b,c,a,d,b)
                \Edges(c,d,x,a)
            \end{scope}
        \end{tikzpicture}
        \caption{A star-critical graph. Vertex $x$ is star-critical as its removal would cause the stars $\{a\}\{b,x\}$ and $\{d\}\{c,x\}$ to not intersect.\label{fig:example_critical}}
\end{figure}
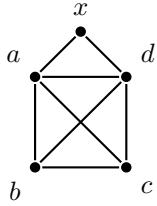

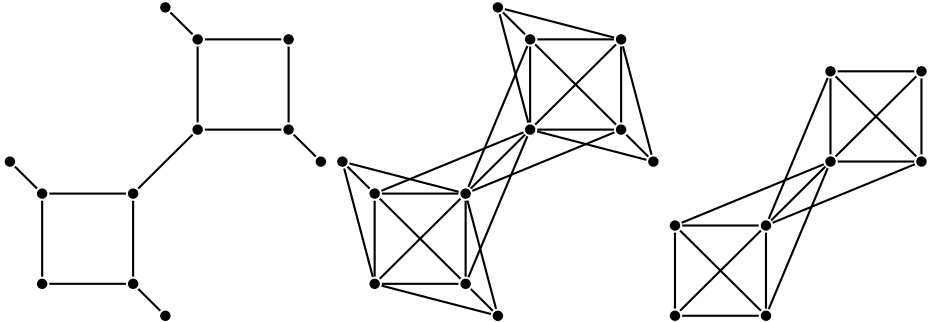
\begin{figure}[!htb]
    \centering
    
        \begin{tikzpicture}[rotate=45,scale=\gscale]
            
                \GraphInit[unit=3,vstyle=Normal]
                \SetVertexNormal[Shape=circle, FillColor=black, MinSize=2pt]
                \tikzset{VertexStyle/.append style = {inner sep = \inners, outer sep = \outers}}
                \begin{scope}[shift={(-2.41cm, 0cm)}]
                    \SetVertexNoLabel
                    \grEmptyCycle[RA=1.41,prefix=a]{4}
                    \Edges(a0,a1,a2,a3,a0)
                    \Vertex[x=0,y=2.41]{1a}
                    \Vertex[x=0,y=-2.41]{3a}
                    \Edge(1a)(a1)
                    \Edge(3a)(a3)
                \end{scope}
                \begin{scope}[shift={(2.41cm, 0cm)}]
                    \SetVertexNoLabel
                    \grEmptyCycle[RA=1.41,prefix=b]{4}
                    \Edges(b0,b1,b2,b3,b0)
                    \Vertex[x=0,y=2.41]{1b}
                    \Vertex[x=0,y=-2.41]{3b}
                    \Edge(1b)(b1)
                    \Edge(3b)(b3)
                \end{scope}
                \Edge(a0)(b2)
        \end{tikzpicture}
        \hfill
        \begin{tikzpicture}[rotate=45,scale=\gscale]
            
                \GraphInit[unit=3,vstyle=Normal]
                \SetVertexNormal[Shape=circle, FillColor=black, MinSize=2pt]
                \tikzset{VertexStyle/.append style = {inner sep = \inners, outer sep = \outers}}
                \begin{scope}[shift={(-2.41cm, 0cm)}]
                    \SetVertexNoLabel
                    \grEmptyCycle[RA=1.41,prefix=a]{4}
                    \Edges(a0,a1,a2,a3,a0)
                    \Vertex[x=0,y=2.41]{1a}
                    \Vertex[x=0,y=-2.41]{3a}
                    \Edge(1a)(a1)
                    \Edge(3a)(a3)
                    
                    \Edge(a0)(a2)
                    \Edge(a1)(a3)
                    
                    \Edge(1a)(a2)
                    \Edge(3a)(a2)
                    \Edge(1a)(a0)
                    \Edge(3a)(a0)
                \end{scope}
                \begin{scope}[shift={(2.41cm, 0cm)}]
                    \SetVertexNoLabel
                    \grEmptyCycle[RA=1.41,prefix=b]{4}
                    \Edges(b0,b1,b2,b3,b0)
                    \Vertex[x=0,y=2.41]{1b}
                    \Vertex[x=0,y=-2.41]{3b}
                    \Edge(1b)(b1)
                    \Edge(3b)(b3)
                    
                    \Edge(b0)(b2)
                    \Edge(b1)(b3)
                    
                    \Edge(1b)(b2)
                    \Edge(3b)(b2)
                    \Edge(1b)(b0)
                    \Edge(3b)(b0)
                \end{scope}
                \Edge(a0)(b2)
                \Edge(a0)(b1)
                \Edge(a0)(b3)
                \Edge(b2)(a1)
                \Edge(b2)(a3)
        \end{tikzpicture}
        \hfill
        \begin{tikzpicture}[shift={(1.41cm,0cm)},rotate=45,scale=\gscale]
            
                \GraphInit[unit=3,vstyle=Normal]
                \SetVertexNormal[Shape=circle, FillColor=black, MinSize=2pt]
                \tikzset{VertexStyle/.append style = {inner sep = \inners, outer sep = \outers}}
                \begin{scope}[shift={(-2.41cm, 0cm)}]
                    \SetVertexNoLabel
                    \grEmptyCycle[RA=1.41,prefix=a]{4}
                    \Edges(a0,a1,a2,a3,a0)
                    \Edge(a0)(a2)
                    \Edge(a1)(a3)
                \end{scope}
                \begin{scope}[shift={(2.41cm, 0cm)}]
                    \SetVertexNoLabel
                    \grEmptyCycle[RA=1.41,prefix=b]{4}
                    \Edges(b0,b1,b2,b3,b0)
                    \Edge(b0)(b2)
                    \Edge(b1)(b3)
                \end{scope}
                \Edge(a0)(b2)
                \Edge(a0)(b1)
                \Edge(a0)(b3)
                \Edge(b2)(a1)
                \Edge(b2)(a3)
        \end{tikzpicture}
        \hfill

    \caption{A triangle-free graph (left), its square (center) and its star graph (right).}
    \label{fig:tri_star}
\end{figure}

Before proceeding to the main results of this work, we make the following remark, which immediately leads us to the property that every star graph of a triangle-free graph is closely related to the square of one of its induced subgraphs.

\begin{observation}
    Every vertex of degree at least two in a triangle free graph is the center of exactly one maximal star.
\end{observation}

\begin{observation}
    If $H$ is a $K_3$-free graph with at least three vertices, where $D$ are its vertices of degree at least two, and $G = \K{S}(H)$, it holds that $G$ is isomorphic to $H[D]^2$.
\end{observation}

As such, every hardness result or polynomial time algorithm for the recognition of squares of triangle-free graphs immediately applies to the class of star graphs of triangle-free graphs.
For an illustration of the previous observation, we refer to Figure~\ref{fig:tri_star}.
For a far more complicated star graph, we refer to Figure~\ref{fig:star_example}.

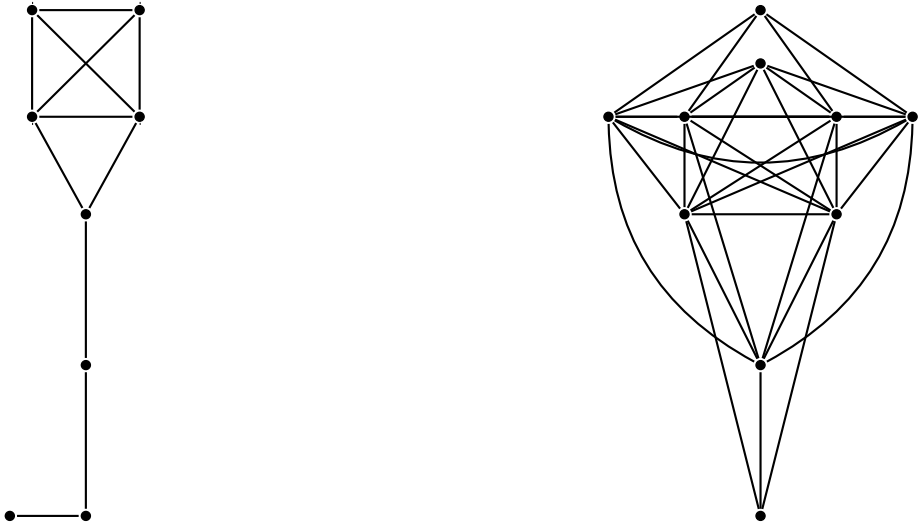
\begin{figure}[!tb]
    \centering
        \begin{tikzpicture}
            \GraphInit[unit=3,vstyle=Normal]
            \SetVertexNormal[Shape=circle, FillColor = black, MinSize=3pt]
            \tikzset{VertexStyle/.append style = {inner sep = \inners,outer sep = \outers}}
            \SetVertexNoLabel
            \begin{scope}[rotate=45]
                \grComplete[RA=1]{4}
            \end{scope}
            \Vertex[x=0, y=-2]{x4}
            \Vertex[x=0, y=-4]{x3}
            \Vertex[x=0, y=-6]{x2}
            \Vertex[x=-1, y=-6]{x1}
            \Edges(a3,x4,a2)
            \Edges(x4,x3,x2,x1)
        \end{tikzpicture}
        \hfill
        \begin{tikzpicture}
            \GraphInit[unit=3,vstyle=Normal]
            \SetVertexNormal[Shape=circle, FillColor = black, MinSize=3pt]
            \tikzset{VertexStyle/.append style = {inner sep = \inners,outer sep = \outers}}
            \SetVertexNoLabel
            \Vertex[x = 0, y = 0.707]{v78}
            \Vertex[x = 2, y = -0.707]{v854}
            \Vertex[x = 1, y = -0.707]{v754}
            
            \Vertex[x = -1, y = -0.707]{v764}
            \Vertex[x = -2, y = -0.707]{v864}
            \Vertex[x = 0, y = 0]{v65}
            
            \Vertex[x = 1, y = -2]{v543}
            
            \Vertex[x = -1, y = -2]{v643}
            
            \Vertex[x = 0, y = -4]{v432}
            
            \Vertex[x = 0, y = -6]{v321}

            \Edge(v65)(v764)
            \Edge(v643)(v854)
            \Edge(v643)(v864)
            \Edge(v321)(v432)
            \Edge(v754)(v854)
            \Edge(v643)(v754)
            \Edge(v543)(v65)
            \Edge(v764)(v78)
            \Edge(v864)(v78)
            \Edge(v543)(v854)
            \Edge(v764)(v754)
            \Edge(v764)(v864)
            \Edge(v432)(v764)
            \Edge(v643)(v65)
            \Edge(v65)(v754)
            \Edge[style = bend right](v864)(v854)
            \Edge(v432)(v643)
            \Edge(v321)(v543)
            \Edge(v643)(v764)
            \Edge(v643)(v543)
            \Edge(v864)(v754)
            \Edge(v754)(v78)
            \Edge(v854)(v78)
            \Edge(v543)(v864)
            \Edge(v543)(v764)
            \Edge(v321)(v643)
            \Edge(v432)(v543)
            \Edge(v65)(v854)
            \Edge(v543)(v754)
            \Edge(v764)(v854)
            \Edge[style=bend right](v432)(v854)
            \Edge[style=bend left](v432)(v864)
            \Edge(v432)(v754)
            \Edge(v65)(v864)
        \end{tikzpicture}
        \hfill

    \caption{A graph (left) and its star graph (right).}
    \label{fig:star_example}
\end{figure}
\section{A bound for star-critical pre-images}

Our first result is an upper bound on the number of vertices of a star-critical graph in terms of its number of maximal stars.
For a graph $H$ the difference $|V(H)| - |\str(H)|$ could be arbitrarily large, but some vertices of $H$ would have to be non-star-critical for such a property to occur (e.g. if $H \simeq K_{1,r}$ there are $r-1$ non-star-critical vertices).
In a sense, star-critical graphs are minimal with respect to the star graph obtained with the application of the star operator.
Recall that the maximal star $s_a$ absorbs the maximal star $s_b$ if, by removing one leaf of $s_b$, it becomes a substar of $s_a$.

\begin{theorem}
    \label{thm:bound}
    If $H$ is an $n$-vertex star-critical graph, $n \leq \frac{1}{2}\left(3|\str(H)|^2 - |\str(H)|\right)$.
\end{theorem}

\begin{proof}
    We begin by partitioning $V(H)$ in $K = \bigcup_{s_a \in \str(H)} \{c(s_a)\}$ and $I = V(H) \setminus K$, which is a subset of its simplicial vertices.
    Note that $I$ is an independent set of $H$, otherwise there would be an edge with endpoints $\{u,v\} \subseteq I$ and either $u$ or $v$ would be in $K$.
    $I$ is further partitioned in $I_A$ and $I_E$: a vertex is in $I_A$ if its removal causes the absorption of at least one star, while the removal of a vertex in $I_E$ causes the disappearance of at least one edge of the star graph.
    
    Note that $|K| \leq |\str(H)|$ holds because each maximal star has a center.
    To bound $|I|$, we divide the analysis in the two situations where a vertex is star-critical.
    \begin{enumerate}
        \item Suppose that the removal of some $z \in I_A$ causes $s_a$, with $u = c(s_a)$, to be absorbed by $s_b$.
        One of two possibilities arise: if $z$ has only one neighbor then $z$ is the only neighbor of $u$ with this property; therefore there are at most $|K|$ such vertices.
        Otherwise, if $z$ has at least two neighbors, there is some $v \in N(z) \cap N(u)$ with $v \in s_b \setminus s_a$. However, since $I$ is an independent set, $v \in K$ and, moreover, $u,z$ are the only neighbors of $v$ in $s_a$, otherwise $s_a \setminus \{z\}$ cannot be a substar of $s_b$. Therefore, for each maximal star $s_a$, since $H$ is star-critical, there is at most one different $z \in I_A$ for each $v \in (N(u) \cap N(z) \cap K) \setminus s_a \subseteq K$ preventing $v$ from being added to $s_a$.
        This implies that the number of vertices required to avoid absorption is at most $|\str(H)|(|K \setminus \{u\}|) \leq 2\binom{|\str(H)|}{2}$.
        \item For the other condition, each $z \in I_E$ could be responsible for the intersection of a different pair of stars of $H$; i.e., there exists $s_a,s_b \in \str(H)$ such that $s_a \cap s_b = \{z\}$.
        Since we have $\binom{|\str(H)|}{2}$ pairs, we may have as many vertices in $I_E$.
    \end{enumerate}
    Summing both cases, we have $|I| \leq 3\binom{|\str(H)|}{2}$ and since $n = |K| + |I|$, it holds that $n \leq \frac{3|\str(H)|^2 - |\str(H)|}{2}$.
\end{proof}

\begin{corollary}
    If $H$ is star-critical and has no simplicial vertex, $|V(H)| \leq |\str(H)|$.
    If the only simplicial vertices of $H$ are leaves, $|V(H)| \leq 2|\str(H)|$.
\end{corollary}

\begin{proof}
    The first statement follows directly from the case where $|I|$ is empty in the proof of Theorem~\ref{thm:bound}.
    The second statement is a consequence of the hypothesis that every vertex of $I_A$ has degree one and $I_E = \emptyset$.
\end{proof}

Improvements to the bound given by Theorem~\ref{thm:bound} appear to require a complete characterization of non-star-critical vertices.
Also, a better understanding of vertices that are required only for the intersection of some stars to be non-empty seems necessary in order to approach the problem through induction.
We believe that the bound on the size of the pre-image is actually linear, however our current analysis falls short of it.
In fact, we conjecture that the constant is actually two; i.e. $|V(H)| \leq 2|V(G)|$.
If this result indeed holds, it would configure an important difference from other intersection graphs.
For instance, there are clique graphs which require a clique-critical pre-image with a quadratic number of vertices~\cite{clique_critical_alcon}.

\section{Characterization}

Much of the following discussion will be about edge clique covers, a central piece on the characterization of many intersection graph classes.
We denote this family of cliques of $G$ by $\mathcal{Q} = \{Q_1, \dots, Q_n\}$.
The usual strategy in these constructions is to construct a bijection between cliques of the intersection graph and vertices of the pre-image.
Since each vertex $a \in V(G)$ must be a star in $H$, it is intuitive to partition each clique as $Q_i \sim \{Q_i^c, Q_i^f\}$, that is, the vertices $a \in Q_i^c$ correspond to the stars of $G$ with center in $v_i \in V(H)$, while the vertices $a \in Q_i^f$ correspond to the stars of $G$ where $v_i \in V(H)$ is a leaf.
We call such an edge clique cover a \tdef{star-partitioned edge clique cover} of $\mathcal{Q}$.
In a slight abuse of notation, for each $a \in V(G)$, we also denote its \tdef{center}, i.e. the unique $i$ such that $a \in Q_i^c$, by $c(a)$, its \tdef{leaf set} by $F(a) = \{i \mid a \in Q_i^f\}$ and its \tdef{cover} by $Q(a) = F(a) \cup \{c(a)\}$. For each pair of cliques $Q_i, Q_j \in \mathcal{Q}$, their \tdef{leaf-leaf intersection} is given by $\ff(i,j) = Q_i^f \cap Q_j^f$ and its \tdef{center-leaf intersection} by $\cf(i,j) = \left(Q_i^c \cap Q_j^f\right) \cup \left(Q_i^f \cap Q_j^c\right)$.

\begin{definition}[Star-compatibility]
    Given a graph $G$ and a star-partitioned edge clique cover $\mathcal{Q}$ of $G$, we say that $\mathcal{Q}$ is \tdef{star-compatible} if, for every $a \in V(G)$, $|Q(a)| \geq 2$, $\exists!\ i$ such that $a \in Q_i^c$ and if, for every $Q_i, Q_j \in \mathcal{Q}$, if $Q_i \cap Q_j \neq \emptyset$, either $\cf(i,j) = \emptyset$ or $\ff(i,j) = \emptyset$.
\end{definition}

\begin{definition}[Star-differentiability]
    \label{def:differentiability}
    Given a graph $G$ and a star-partitioned edge clique cover $\mathcal{Q}$ of $G$, we say that $\mathcal{Q}$ is \tdef{star-differentiable} if for every $Q_i \in \mathcal{Q}$ and for every pair $\{a, a'\} \subseteq Q_i$ the following conditions hold:
    \begin{enumerate}
        \item If $\{a, a'\} \subseteq Q_i^c$, there exists $Q_j, Q_k \in \mathcal{Q}$ such that $a \in Q_j^f$, $a' \in Q_k^f$, $a \notin Q_k^f$, $a' \notin Q_j^f$ and $\cf(j,k) \neq \emptyset$. Moreover, if $Q_i^c \cap Q_j^f \cap Q_k^f = \emptyset$, $\cf(j,k) \neq \emptyset$.
        \item If $a \in Q_i^c$, $a' \in Q_k^c$ and $a \notin Q_k^f$, then there is some $j \in F(a)$ with $\cf(j,k) \neq \emptyset$, $j \notin Q(a')$ and, for every $j' \in F(a)$ with $\cf(j',k) = \emptyset$, $Q_i^c \cap \bigcap_{j'} \ff(j',k) \neq \emptyset$.
        \item If $a \in Q_i^c$, $a' \in Q_k^c$ and $a \in Q_k^f$, for every $j \in F(a) \setminus \{k\}$, $\cf(j,k) = \emptyset$.
        \item If $\{a, a'\} \subseteq Q_i^f$ and $j = c(a) \neq c(a') = k$, then either $Q_i^c \cap \ff\left(j,k\right) \neq \emptyset$ or $\cf\left(j,k\right) \neq \emptyset$.
    \end{enumerate}
\end{definition}

Figures~\ref{fig:diff_cases} and~\ref{fig:diff_cases2} show the four cases of Definition~\ref{def:differentiability} as seen on the pre-image of the star graph we build from $\mathcal{Q}$ during the proof of Theorem~\ref{thm:star_characterization}.

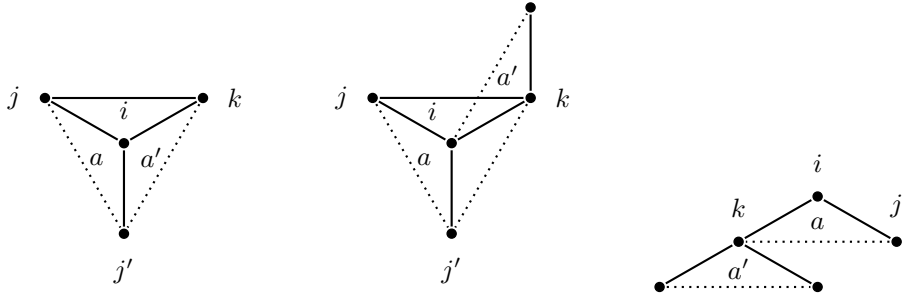
\begin{figure}[!htb]
    \centering
            \begin{tikzpicture}[scale=\gscale]
                \begin{scope}[shift={(0cm,1cm)}]
                    \GraphInit[unit=3,vstyle=Normal]
                    \SetVertexNormal[Shape=circle, FillColor=black, MinSize=2pt]
                    \tikzset{VertexStyle/.append style = {inner sep = \inners, outer sep = \outers}}
                    \Vertex[Math,Ldist=3pt,Lpos=90,LabelOut,L={i},x=0,y=0]{i1}
                    \Vertex[Math,Ldist=3pt,Lpos=180,LabelOut,L={j},a=150, d=2]{j1}
                    \Vertex[Math,Ldist=3pt,Lpos=0,LabelOut,L={k},a=30, d=2]{k1}
                    \node[] at (-0.6,-0.35) {$a$};
                    \node[] at (0.6,-0.31) {$a'$};
                    \Vertex[Math,x=0,y=-2,Ldist=3pt,Lpos=270,LabelOut,L={j'}]{r1};
                    
                    \Edge(i1)(j1)
                    \Edge(i1)(r1)
                    \Edge(i1)(k1)
                    \Edge(j1)(k1)
                    \Edge[style=dotted](j1)(r1)
                    \Edge[style=dotted](k1)(r1)
                \end{scope}
            \end{tikzpicture}
    \hfill
            \begin{tikzpicture}[scale=\gscale]
                \begin{scope}[shift={(0cm,0cm)}]
                    
                    \GraphInit[unit=3,vstyle=Normal]
                    \SetVertexNormal[Shape=circle, FillColor=black, MinSize=2pt]
                    \tikzset{VertexStyle/.append style = {inner sep = \inners, outer sep = \outers}}
                    \Vertex[Math,Ldist=3pt,Lpos=110,LabelOut,L={i},x=0,y=0]{i2}
                    \Vertex[Math,Ldist=3pt,Lpos=180,LabelOut,L={j},a=150, d=2]{j2}
                    \Vertex[Math,Ldist=3pt,Lpos=0,LabelOut,L={k},a=30, d=2]{k2}
                    \node[] at (-0.6,-0.35) {$a$};
                    \node[] at (1.23,1.50) {$a'$};
                    \Vertex[Math,Ldist=3pt,Lpos=270,LabelOut,L={j'},x=0,y=-2]{j'2};
                    \SetVertexNoLabel
                    \Vertex[x=1.73,y=3]{r2}
                    
                    \Edge(i2)(j2)
                    \Edge(i2)(j'2)
                    \Edge(i2)(k2)
                    \Edge(j2)(k2)
                    \Edge(r2)(k2)
                    \Edge[style=dotted](j2)(j'2)
                    \Edge[style=dotted](k2)(j'2)
                    \Edge[style=dotted](i2)(r2)
                \end{scope}
            \end{tikzpicture}
    \hfill
            \begin{tikzpicture}[rotate=180,scale=\gscale]
                \begin{scope}[shift={(0cm,0cm)}]
                    \GraphInit[unit=3,vstyle=Normal]
                    \SetVertexNormal[Shape=circle, FillColor=black, MinSize=2pt]
                    \tikzset{VertexStyle/.append style = {inner sep = \inners, outer sep = \outers}}
                    \Vertex[Math,Ldist=3pt,Lpos=90,LabelOut,L={i},x=0,y=0]{i3}
                    \Vertex[Math,Ldist=3pt,Lpos=90,LabelOut,L={j},a=150, d=2]{j3}
                    \Vertex[Math,Ldist=3pt,Lpos=90,LabelOut,L={k},a=30, d=2]{k3}
                    \SetVertexNoLabel
                    \Vertex[Math,Ldist=3pt,Lpos=0,LabelOut,L={k},a=30, d=4]{r3}
                    \Vertex[Math,Ldist=3pt,Lpos=0,LabelOut,L={k},x=0, y=2]{l3}
                    \node[] at (0,0.6) {$a$};
                    \node[] at (1.73,1.6) {$a'$};
                    
                    \Edge(i3)(j3)
                    \Edge(i3)(k3)
                    \Edge[style=dotted](j3)(k3)
                    \Edge(k3)(r3)
                    \Edge(k3)(l3)
                    \Edge[style=dotted](r3)(l3)
                \end{scope}
            \end{tikzpicture}
    \hfill

    \caption{The first three cases of Definition~\ref{def:differentiability}, from left (first) to right (third).}
    \label{fig:diff_cases}
\end{figure}

\begin{figure}[!htb]
    \centering
    
        \begin{tikzpicture}[scale=\gscale]
            
            \begin{scope}[shift={(-4cm,-1cm)}, rotate=180]
                \GraphInit[unit=3,vstyle=Normal]
                \SetVertexNormal[Shape=circle, FillColor=black, MinSize=2pt]
                \tikzset{VertexStyle/.append style = {inner sep = \inners, outer sep = \outers}}
                \Vertex[Math,Ldist=3pt,Lpos=90,LabelOut,L={i},x=0,y=0]{i4}
                \Vertex[Math,Ldist=3pt,Lpos=90,LabelOut,L={k_1},a=150, d=2]{j4}
                \Vertex[Math,Ldist=3pt,Lpos=90,LabelOut,L={j_1},a=30, d=2]{k4}

                \SetVertexNoLabel
                \Vertex[Math,Ldist=3pt,Lpos=0,LabelOut,L={k},a=30, d=4]{r4}
                \Vertex[Math,Ldist=3pt,Lpos=0,LabelOut,L={k},x=0.4, y=2]{l4}
                \Vertex[Math,Ldist=3pt,Lpos=0,LabelOut,L={k},a=150, d=4]{m4}
                \Vertex[Math,Ldist=3pt,Lpos=0,LabelOut,L={k},x=-0.4, y=2]{n4}

                \node[] at (-1.77,1.5) {$a_1'$};
                \node[] at (1.77,1.5) {$a_1$};
                
                \Edge(i4)(j4)
                \Edge(i4)(k4)
                \Edge[style=dotted](j4)(k4)
                \Edge(k4)(r4)
                \Edge(k4)(l4)
                \Edge[style=dotted](r4)(l4)
                \Edge(j4)(n4)
                \Edge(j4)(m4)
                \Edge[style=dotted](m4)(n4)

            \end{scope}
        \end{tikzpicture}
        \hfill
        \begin{tikzpicture}[scale=\gscale]
            \begin{scope}[shift={(4cm,1cm)}]
                \GraphInit[unit=3,vstyle=Normal]
                \SetVertexNormal[Shape=circle, FillColor=black, MinSize=2pt]
                \tikzset{VertexStyle/.append style = {inner sep = \inners, outer sep = \outers}}
                \Vertex[Math,Ldist=3pt,Lpos=90,LabelOut,L={i},x=0,y=0]{i5}
                \Vertex[Math,Ldist=3pt,Lpos=90,LabelOut,L={j_2},a=210, d=2]{j5}
                \Vertex[Math,Ldist=3pt,Lpos=90,LabelOut,L={k_2},a=-30, d=2]{k5}

                \SetVertexNoLabel

                \Vertex[Math,Ldist=3pt,Lpos=0,LabelOut,L={k},a=-30, d=4]{r5}
                \Vertex[Math,Ldist=3pt,Lpos=0,LabelOut,L={k},x=0.4, y=-2]{l5}
                \Vertex[Math,Ldist=3pt,Lpos=0,LabelOut,L={k},a=210, d=4]{m5}
                \Vertex[Math,Ldist=3pt,Lpos=0,LabelOut,L={k},x=-0.4, y=-2]{n5}
                
                \node[] at (-1.77,-1.5) {$a_2$};
                \node[] at (1.77,-1.5) {$a'_2$};

                \Edge(i5)(j5)
                \Edge(i5)(k5)
                \Edge(j5)(k5)
                \Edge(k5)(r5)
                \Edge(k5)(l5)
                \Edge[style=dotted](r5)(l5)
                \Edge(j5)(n5)
                \Edge(j5)(m5)
                \Edge[style=dotted](m5)(n5)
                
            \end{scope}
        \end{tikzpicture}

    \caption{The fourth case of Definition~\ref{def:differentiability}.}
    \label{fig:diff_cases2}
\end{figure}
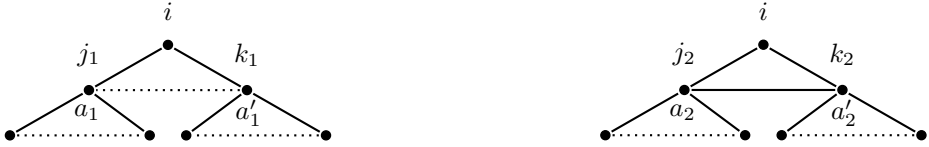

We emphasize that: (i) star-compatibility translates the structural properties of stars; and (ii) star-differentiability enumerates the possible ways that two stars that share at least one vertex are different.
Note that, the \say{missing case}, where $\{a,a'\} \in Q_i^f$ and $c(a) = c(a') = k$ is exactly the same case as 1, but with $\{a,a'\} \in Q_k^c$ instead of $Q_i^c$.

\begin{lemma}
    \label{lem:star_maximality}
    Let $G$ be a graph and $\mathcal{Q}$ a star-partitioned edge clique cover of $G$. If $\mathcal{Q}$ is star-compatible and star-differentiable then, for every pair $\{a, a'\} \subseteq V(G)$, $Q(a) \nsubseteq Q(a')$ and $Q(a') \nsubseteq Q(a)$.
\end{lemma}

\begin{tproof}
    If $a$ and $a'$ do not share any clique, the statement holds.
    Otherwise they do share some clique, say $Q_i$.
    If the pair $a,a'$ satisfies properties 1, 2, or 4 of Definition~\ref{def:differentiability}, since $i \in Q(a) \cap Q(a')$, we conclude that there exists $j \in Q(a)$, $k \in Q(a')$ but $j \notin Q(a')$ and $k \notin Q(a)$, implying $Q(a) \nsubseteq Q(a')$ and $Q(a') \nsubseteq Q(a)$.
    
    For property 3, however, we first conclude that there is some $j \in Q(a)$ but $j \notin Q(a')$, otherwise we would have $\cf(j,k) \neq \emptyset$ and $\ff(j,k) \neq \emptyset$.
    Consequently, $Q(a) \nsubseteq Q(a')$.
    To see that $Q(a') \nsubseteq Q(a)$, note that $\{a, a'\} \subseteq Q_k$ and, following the same argument, we conclude that there is some $j' \in Q(a')$ but $j' \notin Q(a)$, completing the proof.
\end{tproof}

We now present a Krausz-type characterization for the class of star graphs.

\begin{theorem}
    \label{thm:star_characterization}
    An $n$-vertex graph $G$ is the star graph of some graph $H$ if and only if there is a star-compatible and star-differentiable star-partitioned edge clique cover $\mathcal{Q}$ of $G$ with at most $\frac{1}{2}(3n^2 - n)$ cliques.
\end{theorem}

\begin{tproof}
    In this proof, we assume that $H$ has $m$ vertices, denoted by $v_i$, and that star $s_a \in \str(H)$ corresponds to the vertex $a \in V(G)$.
    
    For the first direction of the statement, assume $H$ is a star-critical pre-image of $G$.
    For each $v_i \in V(H)$, let $S(v_i) = \{s_a \in \str(H) \mid v_i \in s_a\}$, that is, the maximal stars of $H$ that contain $v_i$.
    Clearly, we can partition these sets as $S(v_i) \sim \{S^c(v_i), S^f(v_i)\}$, that is, the stars where $v_i$ is the center and where it is a leaf, respectively.
    Our goal is to show that $\mathcal{Q} = \{Q_1, \dots, Q_m\}$, with $Q_i^c = S^c(v_i)$ and $Q_i^f = S^f(v_i)$ is a star-partitioned edge clique cover of $G$ satisfying star-compatibility and star-differentiability which.
    By Theorem~\ref{thm:bound}, this is all that remains is to be proven, since $|\mathcal{Q}| = |V(H)| \leq \frac{1}{2}(3n^2 - n)$.
    
    To verify that $\mathcal{Q}$ is a star-partitioned edge clique cover of $G$, first note that every $Q_i$ is a clique of $G$, since the corresponding stars share at least $v_i \in V(H)$.
    For the coverage part, every $aa' \in E(G)$ has two corresponding stars $s_a, s_{a'} \in \str(H)$, which share at least one vertex, say $v_i \in V(H)$, since $G \simeq \K{S}(H)$.
    By the construction of $\mathcal{Q}$, there is some $Q_i \in \mathcal{Q}$ which corresponds to every maximal star that contains $v_i$; this guarantees that $aa'$ is covered by at least one clique of $\mathcal{Q}$.
    
    For the other properties, first take two vertices $v_i,v_j \in V(H)$ with $v_iv_j \notin E(H)$ but $S(v_i) \cap S(v_j) \neq \emptyset$.
    Clearly, no star in $S(v_i) \cap S(v_j)$ may have $v_i$ and $v_j$ in different sides of its bipartition, thus $S(v_i) \cap S(v_j) = S^f(v_i) \cap S^f(v_j)$.
    Now, suppose that $v_iv_j \in E(H)$; since they are adjacent, any star in $S(v_i) \cap S(v_j)$ must have $v_i$ and $v_j$ in opposite sides of the bipartition and, thus, we have that $S(v_i) \cap S(v_j) = \left(S^c(v_i) \cap S^f(v_j)\right) \cup \left(S^f(v_i) \cap S^c(v_j)\right)$.
    Since each star has a single center, the above analysis shows that $\mathcal{Q}$ satisfies star-compatibility.
    
    For star-differentiability, let $\{s_a, s_{a'}\} \subseteq S(v_i)$. We break our analysis in the same order as the one given in Definition~\ref{def:differentiability}.
    \begin{enumerate}
        \item If $\{s_a, s_{a'}\} \subseteq S^c(v_i)$ there must be at least one leaf in each star, say $v_j$ and $v_k$, respectively, not in the other and these leaves must be adjacent to each other, otherwise at least one of the stars would not be maximal.
        That is, $\{a, a'\} \in Q_i^c$ imply that there is $Q_j,Q_k \in \mathcal{Q}$ with $a \in Q_j^f$, $a' \in Q_k^f$, $a \notin Q_k^f$, $a \notin Q_j^f$ and $\cf(j,k) \neq \emptyset$.
        \item If $s_a \in S(v_i)^c$, $s_{a'} \in S(v_k)^c$ and $s_a \notin S(v_k)^f$, $v_iv_k \in E(H)$ and to keep $v_k$ from being a leaf of $s_a$, one leaf of $s_a$, say $v_j$, must also be adjacent to $v_k$ and not a leaf of $s_{a'}$, since $v_i$ is.
        Now, for every $v_{j'} \in s_a$ and not adjacent to $v_k$, there is a clear $P_3 = v_kv_iv_{j'}$, which must be part of some maximal star.
        Moreover, the set of all $v_{j'}$ non-adjacent to $v_k$ will form a maximal star centered around $v_i$ along with $v_k$.
        Thus, $a \in Q_i^c$, $a' \in Q_k^c$ and $a \notin Q_k^f$, imply that there is some $j \in F(a)$ with $\cf(j,k) \neq \emptyset$, $j \notin Q(a')$ and, for $j' \in F(a)$ with $\cf(j', k) = \emptyset$, $Q_i^c \cap \bigcap_{j'} \ff(j',k) \neq \emptyset$.
        \item If $s_a \in S(v_i)^c$, $s_{a'} \in S(v_k)^c$ and $s_a \in S(v_k)^f$, we know that $s_a = \{v_i\}\{v_k, \dots\}$ and, since $v_k$ is not adjacent to any other leaf $v_j$ of $s_a$, we know that $S(v_j) \cap S(v_k) = S^f(v_j) \cap S^f(v_k)$ and, since $v_k$ is the center of $s_{a'}$, $v_j$ is not one of its leaves.
        Therefore, $a \in Q_i^c$, $a' \in Q_k^c$ and $a \in Q_k^f$, implies that for every $j \in F(a) \setminus \{k\}$, $\cf(j,k) = \emptyset$.
        \item If $\{s_a, s_{a'}\} \subseteq Q_i^f$ and $s_a \in S^c(v_j)$, $s_{a'} \in S^c(v_k)$, either $v_jv_k \notin E(H)$, which induces the existence a star $\{v_i\}\{v_j, v_k, \dots\}$, or $v_jv_k \in E(H)$, which must be part of a star with either $v_j$ or $v_k$ as center and the other as a leaf.
        Hence, $\{a, a'\} \subseteq Q_i^f$ and $j = c(a) \neq c(a') = k$, implies that either $Q_i^c \cap \ff\left(j,k\right) \neq \emptyset$ or $\cf\left(j,k\right) \neq \emptyset$.
    \end{enumerate}
    The above shows that $\mathcal{Q}$ is also star-differentiable, which completes this part of the proof.
    
    For the converse, take $\mathcal{Q}$ a star-partitioned edge clique cover of $G$ satisfying star-compatibility and star-differentiability  of size at most $\frac{1}{2}(3n^2 - n)$ and let $H$ be a graph with $V(H) = \{v_i \mid Q_i \in \mathcal{Q}\}$ and $E(H) = \{v_iv_j \mid \cf(i,j) \neq \emptyset\}$ and let us prove that $G \simeq \K{S}(H)$.
    
    Take $a \in V(G)$ with $c(a) = i$.
    Due to star-compatibility and the construction of $H$, we know that $H[\{v_j \mid j \in F(a)\}]$ is an independent set of $H$ and that $s_a = \{v_i\}\{v_j \mid j \in F(a)\}$ is a star of $H$.
    Suppose, however, that $s_a$ is not maximal, that is, there is some $v_k \in V(H)$ such that $v_iv_k \in E(H)$ and $s_b = s_a \cup \{v_k\}$ is a star of $H$.
    By the construction of $H$, either there is some $a' \in V(G)$ such that $Q(a) \subseteq Q(a')$, which is impossible due to Lemma~\ref{lem:star_maximality}, or some $a' \in \cf(i,k)$, which we analyze below.
    The following is based on the first two cases of Definition~\ref{def:differentiability}; the other two are impossible, since $k \notin Q(a)$ and $a \in Q_i^c$.
    
    \begin{enumerate}
        \item If $a' \in Q_i^c$, there is some $Q_j \in \mathcal{Q}$ such that $a \in Q_j^f$ and $\cf(j,k) \neq \emptyset$, which implies that $v_jv_k \in E(H)$ and $s_b$ is not a star of $H$. 
        \item If $a' \in Q_k^c$ and $a \notin Q_k^f$, at least one $j \in F(a)$ satisfies $\cf(j,k) \neq \emptyset$ and $j \notin Q(a')$.
        This gives us that $v_jv_k \in E(H)$ and $s_b$ is not a star of $H$.
    \end{enumerate}
    
    Therefore, we conclude that $a'$ cannot exist, that $s_a$ is maximal and, consequently that $V(G) \subseteq V(\K{S}(H))$.
    
    To show that $V(\K{S}(H)) \subseteq V(G)$, take $s = \{v_i\}L$, with $s \in \str(H)$, and suppose that there is some $j,k \in L$ and that for every pair $a \in \cf(i,j)$ and $a' \in \cf(i,k)$, $a \notin Q_k$ and $a' \notin Q_j$.
    That is, $Q_i \cap Q_j \cap Q_k = \emptyset$, due to star-compatibility and the hypothesis that $jk \notin E(H)$.
    Once again, we analyze the possibilities in terms of Definition~\ref{def:differentiability}.
    
    \begin{enumerate}
        \item If $c(a) = c(a') = i$, we have that $\cf(j,k)  \neq \emptyset$, implying that $v_jv_k \in E(H)$, contradicting the hypothesis that $s$ exists.
        \item If $c(a) = i$ and $c(a') = k$, there is some $j' \in Q(a)$ with $\cf(j,k) \neq \emptyset$. To conclude that $j = j'$, we note that, if $j \neq j'$, it would be required that $Q_i^c \cap \ff(j,k) \neq \emptyset$, which is impossible since $Q_i \cap Q_j \cap Q_k = \emptyset$.
        Once again, contradicting the hypothesis that such an $s$ exists.
        \item Trivially impossible since $Q_i \cap Q_j \cap Q_k = \emptyset$.
        \item If $j = c(a) \neq c(a') = k$, either $Q_i^c \cap \ff(j,k) \neq \emptyset$, which is impossible since $Q_i \cap Q_j \cap Q_k = \emptyset$, or $\cf(j,k) \neq \emptyset$, implies that $v_jv_k \in E(H)$ and that $s$ is not a star.
    \end{enumerate}
    
    The above allows us to conclude that there is no $s \in \str(H)$ generated by cliques not pairwise intersecting.
    Such intersection has a unique vertex of $G$ in it due to Lemma~\ref{lem:star_maximality}, which allows us to conclude $V(\K{S}(H)) \subseteq V(G)$ and, consequently, that $V(\K{S}(H)) = V(G)$.
    
    To show that $E(G) \subseteq E(\K{S}(H))$, we first take an edge $ab \in E(G)$.
    Since $\mathcal{Q}$ is a star-partitioned edge clique cover of $G$, there is some $i$ such that $\{a,b\} \subseteq Q_i$ and, because $V(G) = V(\K{S}(H))$ and the construction of $H$, there are corresponding stars $s_a, s_b \in \str(H)$ with $v_i \in s_a \cap s_b$ which guarantee that $ab \in E(\K{S}(H))$.
    For $E(\K{S}(H)) \subseteq E(G)$, take two intersecting stars $s_a,s_b \in \str(H)$ and note that, since $a,b \in \K{S}(H) = V(G)$ and $\mathcal{Q}$ is a star-partitioned edge clique cover of $G$, $ab \in E(G)$ and we conclude that $E(G) = E(\K{S}(H))$, completing the proof.
\end{tproof}

We now pose a version of the decision problem for star graph recognition, which we call \textsc{Star Graph Recognition}.
We will further require that the output for any algorithm for \textsc{Star Graph Recognition} is already star-partitioned.

\problem{Star Graph Recognition}{A graph $G$.}{Is there a star-partitioned edge clique cover $\mathcal{Q}$ of $G$ satisfying star-compatibility and star-differentiability?}

Theorem~\ref{thm:verification_alg} provides a straightforward verification algorithm to check if a star-partitioned edge clique cover is star-compatible and star-differentiable.

\begin{theorem}
    \label{thm:verification_alg}
    Given a graph $G$ of order $n$, there is an $\bigO{\max\{n^2m, m^2\}n^2m}$ algorithm to decide if a star-partitioned family $\mathcal{Q} \subseteq 2^{V(G)}$ of size $m$ is an edge clique cover of $G$ satisfying star-compatibility and star-differentiability. 
\end{theorem}

\begin{tproof}
    The first task is to determine whether or not $\mathcal{Q}$ is a star-partitioned edge clique cover of $G$.
    The usual $n^2$ algorithm that tests if each $Q_i$ is a clique suffices.
    To check if $\mathcal{Q}$ is an edge clique cover, for each of the $\bigO{n^2}$ edges, we test if one of the $n$ cliques contains it. 
    This simple test takes $\bigO{n^2m}$ time.
    
    To check for star-compatibility: first, for each vertex $a$ of $G$ and each clique $Q_i$, verify if there is a single $i$ such that $a \in Q_i^c$ and at least one $j$ with $a \in Q_j^f$;
    afterwards, for each pair of intersecting cliques $Q_i, Q_j$, test if $\cf(i,j) = \emptyset$ or $\ff(i,j) = \emptyset$.
    The entire process takes $\bigO{nm^2}$ time.
    
    For star-differentiability, we assume that every pairwise intersection of $\mathcal{Q}$ has already been computed in time $\bigO{nm^2}$, and each query $\cf(j,k)$ and $\ff(j,k)$ takes $\bigO{1}$ time.
    Now, for each clique $Q_i$ and for each pair of vertices $\{a, a'\} \in Q_i$, we must check one of the four conditions as follows.
    \begin{enumerate}
        \item If $c(a) = c(a') = i$, for each pair $j \in Q(a)$, $k \in Q(a')$, check if $a' \notin Q_j^f$, $a \notin Q_k^f$ and $\cf(j,k) \neq \emptyset$; this case takes $\bigO{n^2}$.
        \item If $c(a) = i, c(a') = k$ and $a \notin Q_k^f$, for each $j \in F(a)$, check if either $\cf(j,k) \neq \emptyset$ and $j \notin Q(a')$ or $\cf(j,k) = \emptyset$ and $Q_i^c \cap \ff(j,k) \neq \emptyset$; this case takes $\bigO{n^2m}$ time.
        \item If $c(a) = i, c(a') = k$ and $a \in Q_k^f$, check for each $j \in F(a) \setminus \{k\}$, if $\cf(j,k) = \emptyset$, taking $\bigO{n}$ time.
        \item If $j = c(a) \neq c(a') = k$, we check if $Q_i^c \cap \ff(j,k) \neq \emptyset$ in $\bigO{n}$ time, and if $\cf(j,k)$ in $\bigO{1}$ time.
    \end{enumerate}
    In the worst case scenario, we will spend $\bigO{\max{n^2m, m^2}}$ time for each $Q_i$ and each pair $\{a, a'\} \subseteq Q_i$, of which there are $\bigO{n^2m}$ combinations, and conclude that the whole algorithm takes no more than $\bigO{\max\{n^2m, m^2\}n^2m}$ time.
\end{tproof}

Together with Theorems~\ref{thm:bound} and~\ref{thm:star_characterization}, Theorem~\ref{thm:verification_alg} implies that deciding whether or not a graph is a star graph is in $\mathsf{NP}$.

\begin{theorem}
    \textsc{Star Graph Recognition} is in $\mathsf{NP}$.
\end{theorem}
\section{Properties}
The next theorem uses the known result, due to~\cite{moon_moser}, that a graph of order $n$ has at most $3^{n/3}$ maximal independent sets.

\begin{theorem}
    If $G$ is the star graph of a $n$ vertex graph $H$, then $|V(G)| \leq n3^{\Delta(H)/3}$.
\end{theorem}

\begin{tproof}
    For every $v \in V(H)$, define $H_v = H[N(v)]$ and note that each maximal independent set of $H_v$ might induce a maximal star of $H$ centered around $v$.
    Since $|V(H_v)| \leq \Delta(H)$, we have that $H_v$ has at most $3^{\Delta(H)/3}$ maximal independent sets and, therefore, $H$ has at most $3^{\Delta(H)/3}$ maximal stars centered around $v$.
    Summing for every $v \in V(H)$ we arrive at the $n3^{\Delta(H)/3}$ bound.
\end{tproof}

The observation made in the proof of the previous theorem is quite useful when one wants to generate $\str(H)$.
In fact, we can do that with \textit{polynomial delay}, i.e., the time between outputting two maximal stars is upper bounded by a polynomial on the size of the graph.
To do so, we employ the polynomial delay algorithm for maximal independent sets of Johnson et al.~\cite{independent_poly_delay}.

\begin{theorem}
    There exists an algorithm that, given a graph $H$ on $n$ vertices, generates $\str(H)$ such that the time between the output of two successive members of $\str(H)$ is bounded by a polynomial in $n$.
\end{theorem}

\begin{proof}
    Let $i(n)$ denote the delay between the generation of two maximal independent sets on a graph with $n$ vertices.
    First, we can test for each edge $uv \in E(H)$ if $\{u\}\{v\}$ is a maximal star of $H$: this is the case if and only if $u$ and $v$ are a pair of true twin vertices.
    After this step is done, we have all maximal stars of size two and, since there is a polynomial number of stars of this size, we have polynomial delay.
    Now for each vertex $v \in V(H)$ we use the polynomial delay algorithm of~\cite{independent_poly_delay} to generate all the maximal independent sets of $H_v = H[N(v)]$, discarding all generated independent sets of size one.
    Essentially, this step generates all maximal induced stars of size at least three centered at $v$ and, moreover, the delay between the output of two distinct stars is at most $i(n)n$, since $H_v$ may only have independent sets of size one.
    Finally, this delay of $i(n)n$ may occur for roughly each $H_v$, yielding a total delay of the order $i(n)n^2$.
\end{proof}

\begin{theorem}
    \label{thm:star_cutless}
    If $G$ is a star graph, $G$ has no cut-vertex.
\end{theorem}

\begin{tproof}
    If $|V(G)| \leq 4$, we are done as there are only 5 graphs that satisfy these constraints and none of them contain a cut-vertex.
    They are $K_1$, $K_2$, $K_3$, $K_4$ and $K_4$ with one missing edge (the diamond).
    The first three are trivial, while the last two are shown in Figure~\ref{fig:four_star}.
    
    For graphs with 5 or more vertices, suppose that there is some cut-vertex $x \in G$, that $A,B$ are two of the connected components obtained after removing $x$ from $G$ and take a pair of vertices $a \in V(A) \cap N(x)$, $b \in V(B) \cap N(x)$.
    Suppose now that $G = \K{S}(H)$ for some $H$ and take the stars $s_a, s_b, s_x$ corresponding to $a, b, x$, respectively.
    Since $ab \notin E(G)$ and $ax, bx \in E(G)$, it holds that $s_a \cap s_x \neq \emptyset$ and $s_b \cap s_x \neq \emptyset$ but $s_a \cap s_b = \emptyset$.
    
    If $c(a) = c(x) = i$ and $k = c(b) \neq c(x)$, $s_x$ and $s_b$ share at least one leaf, say $v_j$, since they intercept at some vertex, and $v_j \notin s_a$.
    However, there is no leaf $v_{j'} \in s_a$ adjacent to $v_j$, otherwise there would be an edge $v_jv_{j'} \in E(H)$ and, consequently, some star $s_y$, corresponding to vertex $y \in V(G)$, that keeps $A,B$ connected and intercepts $s_a, s_b, s_x$.
    Therefore, we conclude that no leaf of $s_a$ is adjacent to $v_j$ and, since $c(a) = c(x)$ and $v_iv_j \in E(H)$, we conclude that $v_j \in s_a$, otherwise it would not be maximal, and, consequently, $v_j \in s_a \cap s_b$ and $ab \in E(H)$, which contradicts the hypothesis that $A,B$ are disconnected after removing $x$.
    The case where $c(x) = c(b) \neq c(a)$ follows the exact same argument.
    
    Now if $c(a) \neq c(x) = i$ and $c(x) \neq c(b)$, it is easy to see that $v_i$ cannot be a leaf of both $s_a$ and $s_b$ simultaneously, otherwise $v_i \in s_a \cap s_b$ and $ab \in E(H)$.
    So we have two cases to analyze:
    \begin{enumerate}
        \item If $v_i$ is a leaf of $s_a$, $v_j \in s_x \cap s_b$ and $k = c(b)$, clearly $s_y = \{v_j\}\{v_i, v_k, \dots \}$ is a maximal star of $H$ that intercepts $s_a, s_b, s_x$, keeping $A,B$ from being disconnected.
        The case where $v_i$ is a leaf $s_b$ is the same, and we omit it for brevity.
        \item If $v_i$ is not a leaf of neither $s_a$ nor $s_b$, $c(a) = j$ and $c(b) = k$, we have leaves $v_{j'} \in s_a \cap s_x$ , $v_{k'} \in s_b \cap s_x$ which form at least two intercepting maximal stars, $s_{a'} = \{v_{j'}\}\{v_i, v_j, \dots\}$ and $s_{b'} = \{v_{k'}\}\{v_i, v_k, \dots\}$, such that $s_{a'} \cap s_a \cap s_x \neq \emptyset$ and $s_{b'} \cap s_b \cap s_x \neq \emptyset$.
    \end{enumerate}
    
    These cases allow us to conclude that $A,B$ remains connected no matter the configuration of the intersection of the corresponding stars in $H$.
    Consequently, $x$ cannot exist and we complete the proof.
\end{tproof}

\begin{theorem}
    Every edge of a connected star graph $G$ is contained in at least one triangle if $|V(G)| \geq 3$.
\end{theorem}

\begin{tproof}
    The only connected star graph with 3 vertices is $K_3$, so take $G$ with $|V(G)| \geq 4$.
    Take a pre-image $H$ of $G$, $ab \in E(G)$, $s_a, s_b \in \str(H)$ the corresponding stars to $a, b$, and assume that $ab$ is not contained in any triangle of $G$.
    Since $G$ is connected, there is at least one $x \in V(G)$ adjacent to (w.l.o.g) $a$, but not to $b$, and a corresponding maximal star $s_x$ of $H$.
    Below, we analyze the possible intersections between $s_a$ and $s_b$ and conclude that there is always some star $s_y$ that shares one vertex with $s_a$ and $s_b$.
    \begin{enumerate}
        \item If $c(a) = c(b) = i$ and the center of $s_x$ is a leaf of $s_a$, clearly $v_i$ is not a leaf of $s_x$, otherwise $s_x \cap s_b \neq \emptyset$, therefore there is some leaf $v_j \in s_x$ with $v_iv_j \in E(H)$, which must be part of at least one maximal star $s_y \neq s_x$ of $H$, from which we conclude that $s_a \cap s_b \cap s_y \neq \emptyset$, $s_a \cap s_x \cap s_y \neq \emptyset$ and both $ab$ and $ax$ are in a triangle of $G$.
        \item If $c(a) = c(b) = i$ and a leaf $v_j$ of $s_x$ is a leaf of $s_a$, either the center $v_k$ of $s_x$ is adjacent to $v_i$, in which case $v_iv_k \in E(H)$ and we follow the same argument as in the previous case, or they are not adjacent, implying that there is a maximal star $s_y = \{v_j\}\{v_i, v_k, \dots\}$ which intercepts $s_a, s_b, s_x$, which allows us to conclude that $s_a, s_b, s_x, s_y$ is a clique of $G$.
        \item if $i = c(a) \neq c(b) = k$, there is some leaf $v_j \in s_a \cap s_b$. Clearly, if $v_iv_k \in E(H)$, there is a star that intercepts both $s_a$ and $s_b$;
        otherwise, $v_iv_k \notin E(H)$ and we conclude that $s_y = \{v_j\}\{v_i, v_k, \dots\}$ intercepts $s_a$ and $s_b$ and creates a triangle that contains $ab$.
    \end{enumerate}
\end{tproof}

The previous theorem implies that the minimum degree of any connected star graph on at least three vertices is at least two.
A natural question arises about the vertices of degree two and the structures on the pre-image that generate them.

A \textit{pending-$P_4$} $\{u, v, w, z\}$ is an induced path on four vertices that satisfies $d(u) = 1$, $N(v) = \{u, w\}$, $N(w) = \{v, z\}$ and $N(z)$ is an independent set of $H$.
A \textit{terminal triangle} is a set $\{u, v, z\}$ such that $N[u] = N[v] = \{u,v,z\}$, and no other pair of vertices in $N(z)$ is adjacent.
In both cases, $z$ is called the \textit{anchor} of the structure.
Our next result shows that for nearly all star graphs, their degree two vertices are either generated by pending-$P_4$'s or terminal triangles.

\begin{lemma}
    If $H$ is star-critical, $G = \K{S}(H)$ is connected, and $G$ is not isomorphic to a diamond, then every vertex of degree two of $G$ is generated by a pending-$P_4$,
    or by a terminal triangle.
    Moreover, for every degree two vertex $a \in V(G)$, it holds that $a$ has a neighbor $a'$ which is not adjacent to another vertex of degree two.
\end{lemma}

\begin{proof}
    If $|V(G)| \leq 3$ the result holds, so suppose $|V(G)| \geq 4$, let $a$ be a degree two vertex of $G$ with $N(a) = \{b,d\}$, $s_a$ be the corresponding maximal star of $H$, $c(s_a) = v$, and $u \in s_a$ be one of its leaves.
    Since $d_G(a) = 2$, neither $v$ nor any of its leaves can be contained in any other maximal star of $H$, aside possibly from $b$ and $d$.
    
    Suppose that $s_a = \{u,v\}$.
    In this case, we have that both $u$ and $v$ are true twins, with $w \in N_H(u)$.
    If $u$ is simplicial, $|N_H(u)| = 2$, otherwise $a$ would have more than two neighbors.
    If $N_H(u) \setminus \{v\}$ is not an independent set, it has at least two adjacent vertices $w,z$ forming a $K_4$ with $u$ and $v$; regardless of the neighborhood of $w$ and $z$, at least four distinct maximal stars contain either $u$ or $v$, implying $d_G(a) \geq 3$.
    If $N_H(u) \setminus \{v\}$ is an independent set, we have two options:
    \begin{enumerate}
        \item $N_H(u) \supseteq \{v,w,z\}$, in which case at least one of $w$ or $z$, say $w$, has a neighbor other than $u$ and $v$, since $H$ is star-critical. This configuration, however, generates a $K_5$ in $G$: two stars centered at $w$, $s_a$, one centered at $v$ containing all its neighbors, except $u$, and one centered at $u$ with all its neighbors except $v$.
        \item Otherwise, $N_H(u) = \{v,w\}$. Since $|V(G)| \geq 4$, $w$ necessarily has an additional neighbor. If $N_H(w) \setminus \{u,v\}$ is not an independent set of $H$, $w$ has a pair of adjacent neighbors $x,y$, which are not adjacent to $u$ nor $v$.
        However, note that there are at least four stars centered at $w$ that intersect $s_a$, contradicting the hypothesis that $a$ has only two neighbors in $G$.
   \end{enumerate}
   From the above, we conclude that if $|s_a| = 2$, it corresponds to an edge of a terminal triangle.
   
   On the other hand, suppose now that $s_a \supseteq \{u,v,w\}$, and that $c(s_a) = v$.
   Towards showing that $N_H(v)$ is an independent set, suppose that $v$ has at least one edge in its neighborhood.
   \begin{enumerate}\addtocounter{enumi}{2}
       \item If no such edge is incident to $u$ or $w$, then there are two maximal stars centered at $v$ containing $\{u,v,w\}$ but, in this case, neither $u$ nor $w$ may have a star centered at it and, consequently, one of them is non-star-critical. 
       \item If $w$ is adjacent to some $z \in N_H(v)$ but $zu \notin E(H)$, again there are two stars centered at $v$ ($s_a$ and another one containing $\{u,v,z\}$) both being adjacent to any star that includes the edge $wz$.
       For $s_a$ to have only two neighbors, neither $u$, nor $v,$ nor $w$ may be in another star.
       Since $G$ is connected and has at least four vertices, $z$ must have another neighbor $x$.
       We subdivide our analysis on the neighborhood of $x$:
       \begin{enumerate}
            \item If $xv \in E(H)$, either $x$ or $z$ must be part of another maximal star; actually, $x$ cannot be the center of another star (note that $x$ is part of $s_a$, or it would be in another star that intersects $s_a$), so $z$ must be part of another star, that is, it has a neighbor $y$ not adjacent to $x$; but, in this case, $\{z,x,y\}$ intersects $s_a$, increasing the degree of $a$ to at least three.
            \item So $xv \notin E(H)$ and there is a star centered at $z$ containing $\{v,z,x\}$ which does not contain $w$, this implies that $s_a$ intersects at least three stars.
       \end{enumerate}
       \item If $z$ is adjacent to both $u$ and $w$, $s_a$ already intersects two maximal stars -- one containing $vz$ and another containing $\{u,z,w\}$.
       Note that neither $w$ nor $u$ may have another neighbor, as that would inevitably generate a third star that intersects $s_a$.
       The only possibility would be that $z$ is part of a maximal star that does not contain neither $u$, nor $v$, nor $w$.
       That is, every star that contains $z$ has it as one of its leaves (otherwise we would have leaves adjacent to $u$, $v$, and $w$).
       This implies that $N_H(z) \setminus \{u,v,w\}$ is an independent set.
       However, either $u$ or $w$ is non-star-critical, since its removal does not change the intersection graph, a contradiction.
   \end{enumerate}
   To realize that $N_H(v) = \{u,w\}$, note that at most two of the neighbors of $v$ may have a single star centered at each of them, all others would be of degree one and, consequently, non-star-critical.
   
   We now show that one of the neighbors of $v$ has degree one.
   To see that this is the case, note that, if neither has degree one, both have at least one neighbor not adjacent to $v$ and, thus, centers of maximal stars containing $v$.
   However, neither may be in \textit{any} other star, as this would increase the degree of $s_a$ to more than two, but this is impossible, since at least one of $u$ and $w$ must be in another star for $G$ to have at least four vertices and remain connected.
   For the remainder of the proof, suppose that $u$ has degree one.
   Together with the fact that $v$ only has two neighbors, we conclude that $w$ must be in precisely two maximal stars, one of them with $w$ being its center, since no neighbor of $w$ may be adjacent to $v$.
   This implies that $N_H(w)$ is either an independent set or that it has at most one edge.
   If $N_H(w)$ has an edge $xy$, however, neither $x$ nor $y$ may have a neighbor not adjacent to $w$, otherwise we would have another star containing $w$ and $s_a$ would intersect three stars.
   In this case, $G$ would be precisely a diamond.
   So now we have that $N_H(w)$ is also an independent set and, furthermore, $N_H(w) = \{z, v\}$, since $H$ is star-critical.
   Now, the only way for $w$ to be in more than two stars is if there is more than one star centered at $z$ containing $w$; which is possible only if $z$ is part of a triangle; so we also conclude that $N(z)$ is an independent set.
   This configuration is precisely a pending-$P_4$.
   
   To show that every vertex $a \in V(G)$ with exactly two neighbors $b,d$ has a neighbor not adjacent to another vertex of degree two, suppose that $a$ was generated by a terminal triangle.
   In this case, the two neighbors are stars centered at the anchor of the triangle; however, any star that intersects $s_b$ must necessarily intersect $s_d$, since the symmetric difference between them is precisely the vertices of $s_a$.
   Thus, since $G$ is not a diamond, no degree two vertex may be adjacent to only one of $b$ or $d$.
   On the other hand, if $a$ was generated by a pending-$P_4$, one of its neighbors, say $b$, is not centered at the anchor of the structure; moreover any neighbor of $b$ must also be adjacent to $d$.
   Thus, regardless of the structure that generated $a$, either a degree two vertex touches both $b$ and $d$, or at least one of them is not adjacent to another vertex of degree two.
   Towards a contradiction, suppose that there is some $a' \in V(G)$ of degree two satisfying $N(a) = N(a')$.
   We have three possible cases:
   \begin{enumerate}\addtocounter{enumi}{5}
       \item If both $a$ and $a'$ belong to pending-$P_4$'s. Note that, if $s_a$ is generated by the $P_4$ $\{u,v,w,z\}$, we have that $s_{a'}$ must be formed by the $P_4$ $\{u',v', z, w\}$, since $s_{a'}$ must intersect the star centered at $z$ and the star centered at $w$.
       However, this implies that $H$ is isomorphic to $P_6$, since the degree of every vertex, except $u$ and $u'$, is two, and we have that $G$ is a diamond, contradicting the hypothesis.
       \item If $s_a$ belongs to a pending-$P_4$ $\{u,v,w,z\}$ and $s_{a'}$ belongs to a terminal triangle, the anchor $z$ of the pending-$P_4$ cannot be the same as the anchor of the terminal triangle, as it would violate the requirement that $N(z)$ is an independent set.
       This, however, makes it impossible for $a'$ to be adjacent to the neighborhood of $a$.
       \item If both stars belong to terminal triangles, a similar analysis as the previous case follows.
   \end{enumerate}
   Finally, we conclude that at most one of the neighbors of a degree two vertex has another degree two neighbor.
\end{proof}

\begin{lemma}
    Let $E_2(G) = \{uv \in E(G) \mid d_G(u) = 2\ \text{or}\ d_G(v) = 2\}$ be the set of edges incident to at least one degree two vertex of the star graph $G$. Unless $G$ is isomorphic to a diamond or a triangle $E_2(G) \leq \min\left\{|V(G)| - 1, \frac{4}{7}|E(G)|\right\}$.
    The bound is tight.
\end{lemma}

\begin{proof}
    Let $V_2(G) = \{v \in V(G) \mid d_G(v) = 2\}$.
    By the previous lemma, we have that for each vertex of degree two there is another vertex non-adjacent to another degree two vertex.
    As such, each pair of edges of $E_2(G)$ with a common endpoint is in one-to-one correspondence with a degree two vertex and its exclusive neighbor, i.e., $|E_2(G)| \leq 2|V_2(G)| \leq |V(G)| - 1$.
    For the second case, for each degree two vertex $v$, its exclusive neighbor has at least two other edges, otherwise the non-exclusive neighbor would be a cut-vertex, but these edges may be between exclusive neighbors.
    As such, we have that $|E_2(G)| + \frac{1}{2}|E_2(G)| + \frac{1}{4}|E_2(G)| \leq |E(G)|$, implying $|E_2(G)| \leq \frac{4}{7}|E(G)|$.
    For the tightness of the bounds, the star graph of $P_7$, the gem, satisfies both conditions.
\end{proof}

We conclude this section with a result about the diameter of a star graph.
In fact, when considering the iterated star operator, it appears that the diameter converges to either three or four, depending on the graph from which the process began, even though the sequence formed by the iterated star graphs itself does not seem to converge.
We highlight that the bound of Theorem~\ref{thm:diameter} is tight, as shown by the example of Figure~\ref{fig:diameter}.

\begin{figure}[!htb]
        \centering
        \begin{tikzpicture}[rotate = 0]
            \GraphInit[unit=3,vstyle=Normal]
            \SetVertexNormal[Shape=circle, FillColor = black, MinSize=3pt]
            \tikzset{VertexStyle/.append style = {inner sep = \inners, outer sep = \outers}}
            \SetVertexNoLabel
            \begin{scope}
                \begin{scope}[shift={(-1.5, 0)}]
                    \grComplete[RA=1, prefix=a]{3}
                \end{scope}
                \begin{scope}[shift={(1.5, 0)}]
                    \begin{scope}[rotate=180]
                        \grComplete[RA=1, prefix=b]{3}
                    \end{scope}
                \end{scope}
                \Edge(b0)(a0)
            \end{scope}
            
            \begin{scope}[shift={(5,0)}]
                    \begin{scope}[rotate=-135]
                        \grComplete[RA=1, prefix=c]{4}
                    \end{scope}
                    \Vertex[a=0, d=1.5]{y1}
                    \Vertex[a=180, d=1.5]{y2}
                    \Edge(y2)(c0)
                    \Edge(y1)(c1)
                    \Edge(y1)(c2)
                    \Edge(y2)(c3)
            \end{scope}
                
        \end{tikzpicture}
        \caption{Problematic case of Theorem~\ref{thm:diameter}. The pre-image on the left and star graph on the right.\label{fig:diameter}}
\end{figure}
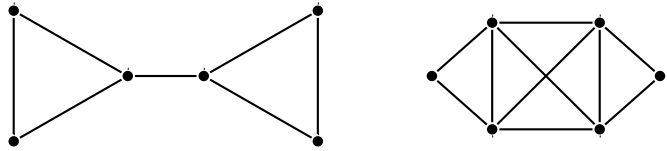

\begin{theorem}
    \label{thm:diameter}
    If $H$ is a graph with diameter $k$ and its star graph $G$ is not a clique, then it holds that the diameter of $G$ is at most $\floor{\frac{k}{2}} + 2$. 
\end{theorem}

\begin{proof}
    Let $P_G = \{s_1, \dots, s_{k+1}\}$ be a diametrical path of $G$. 
    For the following argumentation, we need to guarantee that the endpoints of the path in $G$ have at least two vertices in a shortest path between their corresponding centers in the pre-image $H$.
    Note that, in the case presented in Figure~\ref{fig:diameter}, neither of the degree two vertices of the star graph satisfy the aforementioned condition.
    Let $u = c(s_2)$, $v = c(s_k)$, $P_H$ be a shortest path between $u, v \in V(H)$ of length $r$.
    If $r \geq 2k - 3$, we are done, as we would certainly have a path in $G$ between $u$ and $v$ of length at least $\ceil{\frac{2k-3}{2}} = k-2$ and, by adding stars $s_1$ and $s_{k+1}$, we would have a path of length at least $k$.
    Otherwise, $r < 2k - 3$, which directly implies that there is a path between $s_2$ and $s_k$ of length at most $k-3$, contradicting the hypothesis that $P_G$ is a diametrical path.
\end{proof}

\section{Small star graphs}

By the observations made in Section~\ref{sec:prelim}, star graphs and square graphs are quite similar, and even coincide under specific conditions, which is the case when the pre-image is triangle-free.
A natural question that thus arises is if the classes are actually the same.
To see that there are star graphs which are not square graphs, Figure~\ref{fig:star_not_square} presents a small example of such a graph.

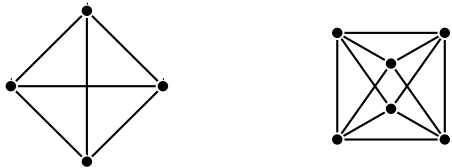
\begin{figure}[!htb]
        \centering
        \begin{tikzpicture}[rotate = 0]
            \GraphInit[unit=3,vstyle=Normal]
            \SetVertexNormal[Shape=circle, FillColor = black, MinSize=3pt]
            \tikzset{VertexStyle/.append style = {inner sep = \inners, outer sep = \outers}}
            \SetVertexNoLabel
            \begin{scope}[shift={(-2, 0)}]
                \grComplete[RA=1]{4}
            \end{scope}
            \begin{scope}[shift={(2, 0)}]
                \begin{scope}[rotate=45]
                    \grCycle[RA=1, prefix=t]{4}
                \end{scope}
                \Vertex[x=0, y=-0.3]{x}
                \Vertex[x=0, y=0.3]{y}
                \Edge(t0)(y)
                \Edge(t1)(y)
                \Edge(t2)(y)
                \Edge(t3)(y)
                \Edge(t0)(x)
                \Edge(t1)(x)
                \Edge(t2)(x)
                \Edge(t3)(x)
            \end{scope}
                
        \end{tikzpicture}
        \caption{The star graph of $K_4$ is not a square graph.\label{fig:star_not_square}}
\end{figure}

Despite not coinciding for many classes of pre-images, it could be the case that every square graph also is a star graph, albeit for a different pre-image.
The smallest example we found of a square graph which is not a star graph is shown in Figure~\ref{fig:square_not_star}: the square of the net.
When attempting to show such a fact, without additional tools the combinatorial explosion of possible cases rapidly becomes intractable.
At the same time, testing all graphs up to the bound given by Theorem~\ref{thm:bound} in search of a pre-image would be completely unfeasible, as we would need to test all connected graphs with up to \textit{51 vertices}.
However, Theorem~\ref{thm:bound} presents what we believe is a very loose value for the size of a pre-image, a claim we support with Theorem~\ref{thm:monotonicity}, its corollary, and some experiments we performed.

\begin{figure}[!htb]
        \centering
        \begin{tikzpicture}[rotate = 0]
            \GraphInit[unit=3,vstyle=Normal]
            \SetVertexNormal[Shape=circle, FillColor = black, MinSize=3pt]
            \tikzset{VertexStyle/.append style = {inner sep = \inners, outer sep = \outers}}
            \SetVertexNoLabel
            \begin{scope}[shift={(-2, 0)}]
                \begin{scope}[rotate=90]
                    \grComplete[RA=.7,prefix=c1]{3}
                    \grEmptyCycle[RA=1.6, prefix=o1]{3}
                    \foreach \i in {0,1,2} {
                        \Edge(c1\i)(o1\i)
                    }
                \end{scope}
            \end{scope}
            \begin{scope}[shift={(2, 0)}]
                \begin{scope}[rotate=90]
                    \grComplete[RA=.7,prefix=c2]{3}
                    \grEmptyCycle[RA=1.6, prefix=o2]{3}
                    \foreach \i in {0,1,2} {
                        \foreach \j in {0,1,2} {
                            \Edge(c2\i)(o2\j)
                        }
                    }
                \end{scope}
            \end{scope}
                
        \end{tikzpicture}
        \caption{The square of the net is not a star graph.\label{fig:square_not_star}}
\end{figure}
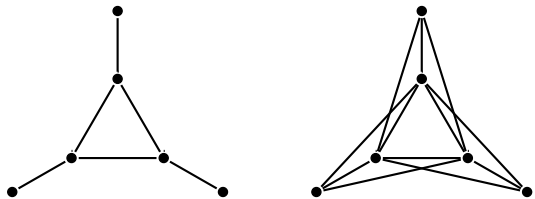

\begin{theorem}
    \label{thm:monotonicity}
    Let $H$ be an $n$-vertex graph with at least one non-star-critical vertex.
    For any $H'$ with $n+1$ vertices such that $H$ is an induced subgraph of $H'$, at least one of the following holds:
    \begin{enumerate}
        \item $H'$ has non-star-critical vertices; or
        \item $|\str(H')| \geq |\str(H)| + 1$.
    \end{enumerate}
\end{theorem}

\begin{proof}
    Since $H$ is a proper induced subgraph of $H'$, let $y$ be the vertex in $V(H') \setminus V(H)$.
    First, note that $|\str(H')| \geq |\str(H)|$ since every maximal induced star of $H$ either remains maximal in $H'$ or is extended to a different maximal induced star of $H'$ by the addition of $y$.
    If $y$ is not a simplicial vertex, at least one star centered at $y$ is lost from $\str(H')$ to $\str(H)$, so condition 2 holds; if $y$ is non-star-critical or if there is some other vertex $x \in V(H)$ that is non-star-critical in $H'$, condition 1 holds.
    So now we may safely assume that $y$ is a star-critical simplicial vertex.
    Suppose that the statement is false, i.e. that every vertex of $H'$ is star-critical and $|\str(H')| = |\str(H)|$; in particular, vertex $x \in V(H)$, which is non-star-critical on $H$, is star-critical in $H'$.
    Before proceeding, note that if $yx \in E(H')$, at least one of $y$ or $x$ must be the center of a star containing this edge and, therefore, we have a new star in $H'$ and condition 2 is satisfied.
    For the remainder of the proof, let $H'' = H \setminus \{x\}$ and $H^* = H'' \cup \{y\}$.
    We divide our analysis in the two cases that make $x$ star-critical in $H'$.
    \begin{enumerate}
        \item Suppose that the removal of $x$ from $H'$ causes the absorption of $s'_a \in \str(H')$ by some $s^*_a \in \str(H^*)$, that is, $(s'_a \setminus \{x\}) \notin \str(H^*)$; this implies that $|\str(H^*)| < |\str(H')|$.
        By the assumption that $|\str(H)| = |\str(H')|$, there is some $s_a \in \str(H)$ satisfying $s_a \subseteq s'_a$, otherwise $s'_a$ would be a new star generated by the addition of $y$.
        Moreover, $(s_a \setminus \{x\}) \in \str(H'')$, since $x$ is non-star-critical in $H$, and $|\str(H'')| = |\str(H)| = |\str(H')|$.
        However, $|\str(H'')| \leq |\str(H^*)| < |\str(H')|$, a contradiction.
        For a clearer view of the double counting involved in this part of the proof, please refer to Figure~\ref{fig:mono}.
        \item There are two stars $s'_a, s'_b \in \str(H')$ such that $s_a' \cap s_b' = \{x\}$.
        Note that, at least one of $s'_a$ and $s'_b$ contains $y$, say $s'_b$, and we have that $(s'_b \setminus \{y\}) \notin \str(H)$, otherwise $s'_a,s'_b \in \str(H)$ and $x$ would be star-critical in $H$.
        Therefore, $s'_b$ is absorbed after the removal of $y$, implying $|\str(H')| > |\str(H)|$.
    \end{enumerate}
\end{proof}

\begin{figure}[!htb]
        \centering
        \begin{tikzpicture}[rotate=0]
            \GraphInit[unit=3,vstyle=Normal]
            \SetVertexNormal[Shape=circle, FillColor = white, MinSize=3pt]
            \tikzset{VertexStyle/.append style = {inner sep = \inners, outer sep = \outers}}
            \Vertex[x=6, y=0, Math, L={H''}]{a}
            \Vertex[x=3, y=1, Math, L={H^*}]{b}
            \Vertex[x=3, y=-1, Math, L={H}]{c}
            \Vertex[x=0, y=0, Math, L={H'}]{d}
            \Edge[style={<-}, label=$- y$](a)(b)
            \Edge[style={<-, double}, label=$- x$](a)(c)
            \Edge[style={<-, dashed}, label=$- x$](b)(d)
            \Edge[style={<-, double}, label=$- y$](c)(d)
        \end{tikzpicture}
        \caption{Relationship between the graphs used in the first case of Theorem~\ref{thm:monotonicity}. The dashed arc indicates that at least one star was absorbed and thick arcs that no stars were absorbed.\label{fig:mono}}
\end{figure}
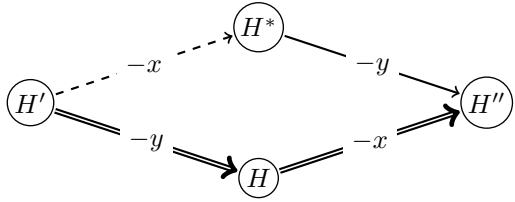

To the best of our knowledge, analogous results to Theorem~\ref{thm:monotonicity} are not known for clique or biclique graphs.
These types of monotonicity properties are particularly useful when looking for small examples; the following statement is a direct corollary.

\begin{corollary}
    \label{cor:frontier}
    Let $G$ be a $k$-vertex graph and $\mathcal{H}_n(k)$ be the set of all graphs on $n$ vertices that have $k$ maximal stars.
    If $G$ is not isomorphic to the star graph of any star-critical $H \in \mathcal{H}_r(k)$ for any $r < n$, and every $H \in \mathcal{H}_n(k)$ is non-star-critical, then $G$ is not a star graph.
\end{corollary}

The above results allowed us to implement a procedure using McKay's Nauty package~\cite{nauty}.
Instead of only looking for the square of the net graph, we generated every star graph on $k \leq 8$ vertices.
In fact, for each $k$, no graph in $\mathcal{H}_{2k+1}(k)$ was star-critical.
Figures~\ref{fig:four_star}, \ref{fig:five_star}, and \ref{fig:six_star} present every star graph on four, five and six vertices, respectively.
There are 46 star graphs on seven vertices, and 201 star graphs on eight vertices.
Let $\mathcal{H}^*(k)$ denote the set of all star-critical pre-images for star graphs on $k$ vertices.
Our procedure also listed $\mathcal{H}^*(k)$ for every $k \leq 8$.
In particular, there are 190 graphs in $\mathcal{H}^*(4)$, 1056 in $\mathcal{H}^*(5)$, 8876 in $\mathcal{H}^*(6)$, 76320 in $\mathcal{H}^*(7)$, and 892170 in $\mathcal{H}^*(8)$.

\begin{figure}
    \centering
    \begin{subfigure}{.5\textwidth}
      \centering
      \begin{tikzpicture}
\GraphInit[unit=3,vstyle=Normal]
\SetVertexNormal[Shape=circle, FillColor = black, MinSize=3pt]
\tikzset{VertexStyle/.append style = {inner sep = \inners,outer sep = \outers}}
\SetVertexNoLabel
\grCycle[RA=1]{4}
\Edge(a0)(a2)
\end{tikzpicture}
    \end{subfigure}%
    \begin{subfigure}{.5\textwidth}
      \centering
      \begin{tikzpicture}
\GraphInit[unit=3,vstyle=Normal]
\SetVertexNormal[Shape=circle, FillColor = black, MinSize=3pt]
\tikzset{VertexStyle/.append style = {inner sep = \inners,outer sep = \outers}}
\SetVertexNoLabel
\grComplete[RA=1]{4}
\end{tikzpicture}
    \end{subfigure}
    \caption{The two four-vertex star graphs. \label{fig:four_star}}
\end{figure}
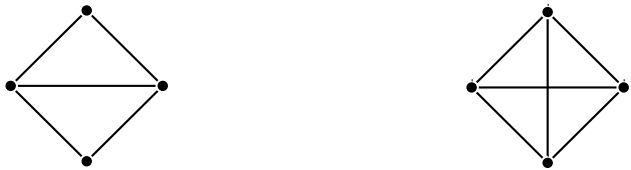

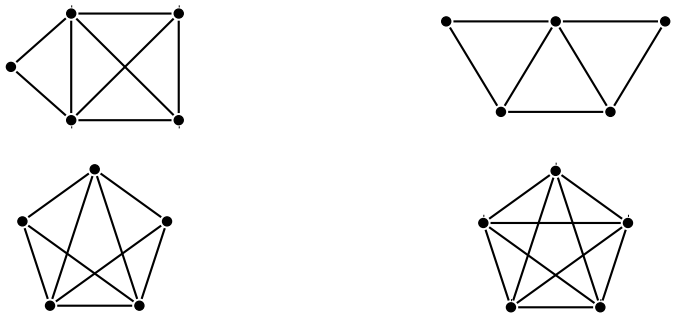
\begin{figure}
    \centering
    \begin{subfigure}{.5\textwidth}
      \centering
      \begin{tikzpicture}[rotate=135]
\GraphInit[unit=3,vstyle=Normal]
\SetVertexNormal[Shape=circle, FillColor = black, MinSize=3pt]
\tikzset{VertexStyle/.append style = {inner sep = \inners,outer sep = \outers}}
\SetVertexNoLabel
\grComplete[RA=1]{4}
\Vertex[a=45, d=1.5]{v}
\Edge(v)(a0)
\Edge(v)(a1)
\end{tikzpicture}
    \end{subfigure}%
    \begin{subfigure}{.5\textwidth}
      \centering
      \begin{tikzpicture}
\GraphInit[unit=3,vstyle=Normal]
\SetVertexNormal[Shape=circle, FillColor = black, MinSize=3pt]
\tikzset{VertexStyle/.append style = {inner sep = \inners,outer sep = \outers}}
\SetVertexNoLabel
\grPath[RA=1.44, RS=1.2, prefix=a]{3}
\begin{scope}[xshift=0.72cm]
    \grPath[RA=1.44, RS=0, prefix=b]{2}
\end{scope}
\Edges(a0,b0,a1,b1,a2)
\end{tikzpicture}
    \end{subfigure}
    \par\bigskip
    \begin{subfigure}{.5\textwidth}
      \centering
      \begin{tikzpicture}[rotate=90]
\GraphInit[unit=3,vstyle=Normal]
\SetVertexNormal[Shape=circle, FillColor = black, MinSize=3pt]
\tikzset{VertexStyle/.append style = {inner sep = \inners,outer sep = \outers}}
\SetVertexNoLabel
\grCycle[RA=1]{5}
\Edges(a4,a2,a0,a3,a1)
\end{tikzpicture}
    \end{subfigure}%
    \begin{subfigure}{.5\textwidth}
      \centering
      \begin{tikzpicture}[rotate=90]
\GraphInit[unit=3,vstyle=Normal]
\SetVertexNormal[Shape=circle, FillColor = black, MinSize=3pt]
\tikzset{VertexStyle/.append style = {inner sep = \inners,outer sep = \outers}}
\SetVertexNoLabel
\grComplete[RA=1]{5}
\end{tikzpicture}
    \end{subfigure}
    \caption{The four five-vertex star graphs.\label{fig:five_star}}
\end{figure}

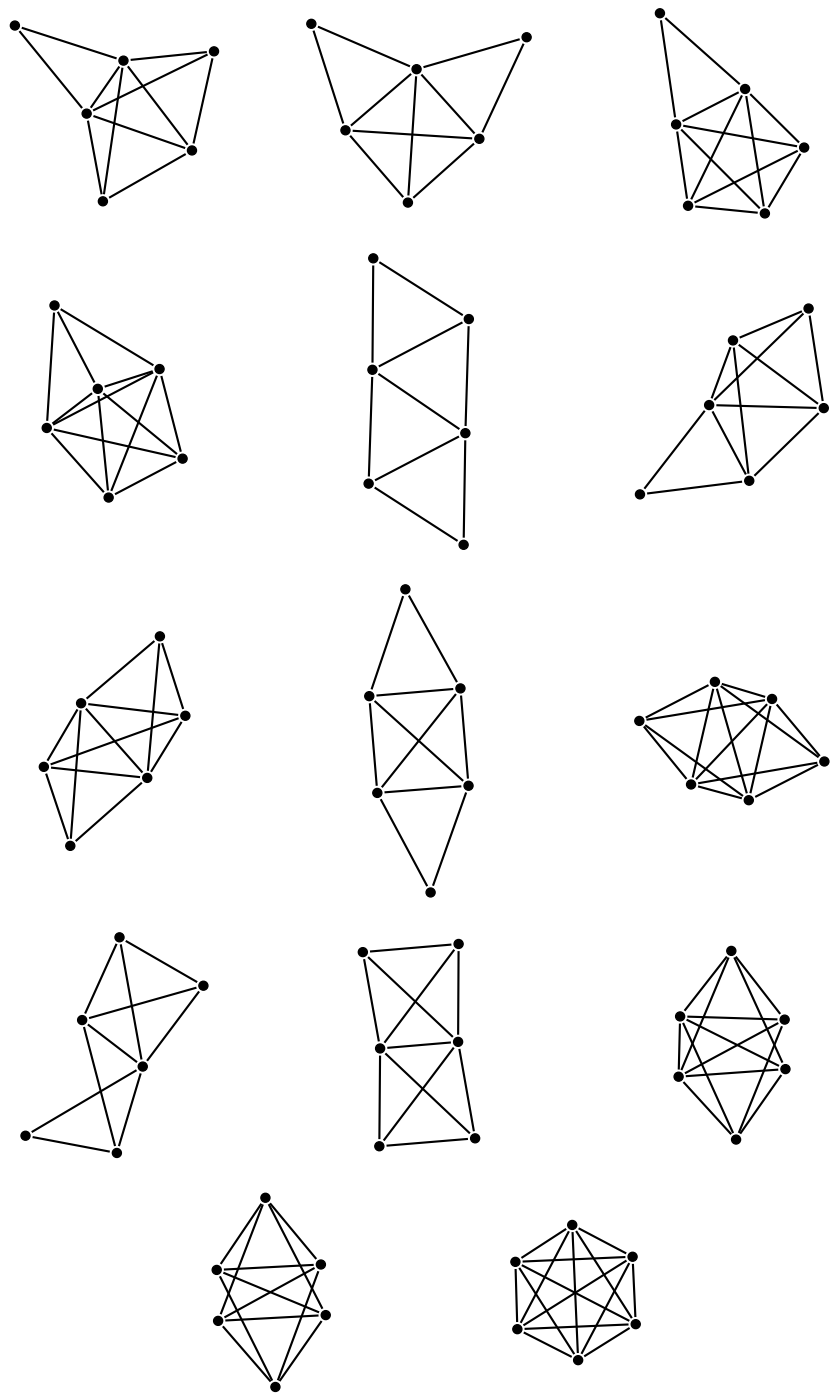
\begin{figure}
    \centering
    \begin{subfigure}{.33\textwidth}
      \centering
      \begin{tikzpicture}
\GraphInit[unit=3,vstyle=Normal]
\SetVertexNormal[Shape=circle, FillColor = black, MinSize=3pt]
\tikzset{VertexStyle/.append style = {inner sep = \inners,outer sep = \outers}}
\SetVertexNoLabel
\Vertex[x = 2.365, y = 0.71044]{v1}
\Vertex[x = 1.19352, y = 0.036]{v0}
\Vertex[x = 0.97988, y = 1.19874]{v3}
\Vertex[x = 2.655, y = 2.0248]{v2}
\Vertex[x = 0.036, y = 2.367]{v5}
\Vertex[x = 1.4641, y = 1.90164]{v4}
\Edge(v0)(v3)
\Edge(v4)(v5)
\Edge(v2)(v3)
\Edge(v3)(v5)
\Edge(v0)(v4)
\Edge(v1)(v4)
\Edge(v1)(v2)
\Edge(v3)(v4)
\Edge(v0)(v1)
\Edge(v1)(v3)
\Edge(v2)(v4)
\end{tikzpicture}
    \end{subfigure}%
    \begin{subfigure}{.33\textwidth}
      \centering
      \begin{tikzpicture}
\GraphInit[unit=3,vstyle=Normal]
\SetVertexNormal[Shape=circle, FillColor = black, MinSize=3pt]
\tikzset{VertexStyle/.append style = {inner sep = \inners,outer sep = \outers}}
\SetVertexNoLabel
\Vertex[x = 2.2458, y = 0.88158]{v1}
\Vertex[x = 1.30628, y = 0.036]{v0}
\Vertex[x = 2.8674, y = 2.2288]{v3}
\Vertex[x = 0.48494, y = 0.9946]{v2}
\Vertex[x = 1.42008, y = 1.80488]{v5}
\Vertex[x = 0.036, y = 2.4068]{v4}
\Edge(v4)(v5)
\Edge(v0)(v2)
\Edge(v0)(v5)
\Edge(v3)(v5)
\Edge(v1)(v2)
\Edge(v1)(v5)
\Edge(v0)(v1)
\Edge(v2)(v5)
\Edge(v1)(v3)
\Edge(v2)(v4)
\end{tikzpicture}
    \end{subfigure}
    \begin{subfigure}{.33\textwidth}
      \centering
      \begin{tikzpicture}
\GraphInit[unit=3,vstyle=Normal]
\SetVertexNormal[Shape=circle, FillColor = black, MinSize=3pt]
\tikzset{VertexStyle/.append style = {inner sep = \inners,outer sep = \outers}}
\SetVertexNoLabel
\Vertex[x = 1.93282, y = 0.90976]{v1}
\Vertex[x = 0.40616, y = 0.137682]{v0}
\Vertex[x = 1.15636, y = 1.68636]{v3}
\Vertex[x = 1.4162, y = 0.036]{v2}
\Vertex[x = 0.036, y = 2.6908]{v5}
\Vertex[x = 0.2489, y = 1.21628]{v4}
\Edge(v0)(v3)
\Edge(v2)(v3)
\Edge(v0)(v2)
\Edge(v4)(v5)
\Edge(v3)(v5)
\Edge(v1)(v3)
\Edge(v0)(v4)
\Edge(v1)(v4)
\Edge(v3)(v4)
\Edge(v0)(v1)
\Edge(v1)(v2)
\Edge(v2)(v4)
\end{tikzpicture}
    \end{subfigure}
    \par\bigskip
    \begin{subfigure}{.33\textwidth}
      \centering
      \begin{tikzpicture}
\GraphInit[unit=3,vstyle=Normal]
\SetVertexNormal[Shape=circle, FillColor = black, MinSize=3pt]
\tikzset{VertexStyle/.append style = {inner sep = \inners,outer sep = \outers}}
\SetVertexNoLabel
\Vertex[x = 1.82262, y = 0.55208]{v1}
\Vertex[x = 0.036, y = 0.95892]{v0}
\Vertex[x = 0.138028, y = 2.5836]{v3}
\Vertex[x = 0.84994, y = 0.036]{v2}
\Vertex[x = 0.70724, y = 1.4778]{v5}
\Vertex[x = 1.5197, y = 1.74048]{v4}
\Edge(v0)(v3)
\Edge(v4)(v5)
\Edge(v0)(v2)
\Edge(v0)(v5)
\Edge(v3)(v5)
\Edge(v0)(v4)
\Edge(v1)(v4)
\Edge(v3)(v4)
\Edge(v1)(v5)
\Edge(v0)(v1)
\Edge(v2)(v5)
\Edge(v1)(v2)
\Edge(v2)(v4)
\end{tikzpicture}
    \end{subfigure}%
    \begin{subfigure}{.33\textwidth}
      \centering
      \begin{tikzpicture}
\GraphInit[unit=3,vstyle=Normal]
\SetVertexNormal[Shape=circle, FillColor = black, MinSize=3pt]
\tikzset{VertexStyle/.append style = {inner sep = \inners,outer sep = \outers}}
\SetVertexNoLabel
\Vertex[x = 1.30832, y = 1.51662]{v1}
\Vertex[x = 1.28532, y = 0.036]{v0}
\Vertex[x = 0.036, y = 0.84724]{v3}
\Vertex[x = 1.3539, y = 3.0298]{v2}
\Vertex[x = 0.097124, y = 3.835]{v5}
\Vertex[x = 0.08617, y = 2.357]{v4}
\Edge(v0)(v3)
\Edge(v4)(v5)
\Edge(v1)(v4)
\Edge(v1)(v2)
\Edge(v3)(v4)
\Edge(v0)(v1)
\Edge(v2)(v5)
\Edge(v1)(v3)
\Edge(v2)(v4)
\end{tikzpicture}
    \end{subfigure}
    \begin{subfigure}{.33\textwidth}
      \centering
      \begin{tikzpicture}
\GraphInit[unit=3,vstyle=Normal]
\SetVertexNormal[Shape=circle, FillColor = black, MinSize=3pt]
\tikzset{VertexStyle/.append style = {inner sep = \inners,outer sep = \outers}}
\SetVertexNoLabel
\Vertex[x = 2.4532, y = 1.17998]{v1}
\Vertex[x = 1.25934, y = 2.0774]{v0}
\Vertex[x = 2.2534, y = 2.5008]{v3}
\Vertex[x = 1.47314, y = 0.2163]{v2}
\Vertex[x = 0.94582, y = 1.22042]{v5}
\Vertex[x = 0.036, y = 0.036]{v4}
\Edge(v0)(v3)
\Edge(v4)(v5)
\Edge(v0)(v2)
\Edge(v0)(v5)
\Edge(v1)(v3)
\Edge(v1)(v5)
\Edge(v0)(v1)
\Edge(v2)(v5)
\Edge(v1)(v2)
\Edge(v3)(v5)
\Edge(v2)(v4)
\end{tikzpicture}
    \end{subfigure}
    \par\bigskip
    \begin{subfigure}{.33\textwidth}
      \centering
      \begin{tikzpicture}
\GraphInit[unit=3,vstyle=Normal]
\SetVertexNormal[Shape=circle, FillColor = black, MinSize=3pt]
\tikzset{VertexStyle/.append style = {inner sep = \inners,outer sep = \outers}}
\SetVertexNoLabel
\Vertex[x = 1.89708, y = 1.76094]{v1}
\Vertex[x = 0.036, y = 1.08372]{v0}
\Vertex[x = 0.38382, y = 0.036]{v3}
\Vertex[x = 1.56226, y = 2.8144]{v2}
\Vertex[x = 1.39586, y = 0.93704]{v5}
\Vertex[x = 0.526, y = 1.92696]{v4}
\Edge(v0)(v3)
\Edge(v4)(v5)
\Edge(v0)(v5)
\Edge(v3)(v5)
\Edge(v0)(v4)
\Edge(v1)(v4)
\Edge(v3)(v4)
\Edge(v1)(v5)
\Edge(v0)(v1)
\Edge(v2)(v5)
\Edge(v1)(v2)
\Edge(v2)(v4)
\end{tikzpicture}
    \end{subfigure}%
    \begin{subfigure}{.33\textwidth}
      \centering
      \begin{tikzpicture}
\GraphInit[unit=3,vstyle=Normal]
\SetVertexNormal[Shape=circle, FillColor = black, MinSize=3pt]
\tikzset{VertexStyle/.append style = {inner sep = \inners,outer sep = \outers}}
\SetVertexNoLabel
\Vertex[x = 1.34148, y = 1.45118]{v1}
\Vertex[x = 0.8426, y = 0.036]{v0}
\Vertex[x = 0.036, y = 2.6392]{v3}
\Vertex[x = 0.141848, y = 1.35614]{v2}
\Vertex[x = 0.51116, y = 4.056]{v5}
\Vertex[x = 1.23518, y = 2.742]{v4}
\Edge(v2)(v3)
\Edge(v0)(v2)
\Edge(v4)(v5)
\Edge(v3)(v5)
\Edge(v1)(v3)
\Edge(v1)(v4)
\Edge(v3)(v4)
\Edge(v0)(v1)
\Edge(v1)(v2)
\Edge(v2)(v4)
\end{tikzpicture}
    \end{subfigure}
    \begin{subfigure}{.33\textwidth}
      \centering
      \begin{tikzpicture}
\GraphInit[unit=3,vstyle=Normal]
\SetVertexNormal[Shape=circle, FillColor = black, MinSize=3pt]
\tikzset{VertexStyle/.append style = {inner sep = \inners,outer sep = \outers}}
\SetVertexNoLabel
\Vertex[x = 1.7815, y = 1.37864]{v1}
\Vertex[x = 0.71638, y = 0.24446]{v0}
\Vertex[x = 1.02986, y = 1.60574]{v3}
\Vertex[x = 1.4757, y = 0.036]{v2}
\Vertex[x = 0.036, y = 1.08744]{v5}
\Vertex[x = 2.4698, y = 0.54924]{v4}
\Edge(v0)(v3)
\Edge(v2)(v3)
\Edge(v0)(v2)
\Edge(v0)(v5)
\Edge(v1)(v3)
\Edge(v0)(v4)
\Edge(v1)(v4)
\Edge(v3)(v4)
\Edge(v1)(v5)
\Edge(v0)(v1)
\Edge(v2)(v5)
\Edge(v1)(v2)
\Edge(v3)(v5)
\Edge(v2)(v4)
\end{tikzpicture}
    \end{subfigure}
    \par\bigskip
    \begin{subfigure}{.33\textwidth}
      \centering
      \begin{tikzpicture}
\GraphInit[unit=3,vstyle=Normal]
\SetVertexNormal[Shape=circle, FillColor = black, MinSize=3pt]
\tikzset{VertexStyle/.append style = {inner sep = \inners,outer sep = \outers}}
\SetVertexNoLabel
\Vertex[x = 1.2723, y = 2.8958]{v1}
\Vertex[x = 0.036, y = 0.2646]{v0}
\Vertex[x = 1.57788, y = 1.18194]{v3}
\Vertex[x = 2.3758, y = 2.254]{v2}
\Vertex[x = 0.78102, y = 1.79834]{v5}
\Vertex[x = 1.23758, y = 0.036]{v4}
\Edge(v0)(v3)
\Edge(v4)(v5)
\Edge(v2)(v3)
\Edge(v1)(v3)
\Edge(v0)(v4)
\Edge(v1)(v2)
\Edge(v1)(v5)
\Edge(v3)(v4)
\Edge(v2)(v5)
\Edge(v3)(v5)
\end{tikzpicture}
    \end{subfigure}%
    \begin{subfigure}{.33\textwidth}
      \centering
      \begin{tikzpicture}
\GraphInit[unit=3,vstyle=Normal]
\SetVertexNormal[Shape=circle, FillColor = black, MinSize=3pt]
\tikzset{VertexStyle/.append style = {inner sep = \inners,outer sep = \outers}}
\SetVertexNoLabel
\Vertex[x = 1.51282, y = 0.14127]{v1}
\Vertex[x = 0.25296, y = 0.036]{v0}
\Vertex[x = 1.29558, y = 2.7186]{v3}
\Vertex[x = 0.26328, y = 1.33302]{v2}
\Vertex[x = 1.28896, y = 1.42044]{v5}
\Vertex[x = 0.036, y = 2.6104]{v4}
\Edge(v2)(v3)
\Edge(v0)(v2)
\Edge(v4)(v5)
\Edge(v0)(v5)
\Edge(v3)(v5)
\Edge(v3)(v4)
\Edge(v1)(v5)
\Edge(v0)(v1)
\Edge(v2)(v5)
\Edge(v1)(v2)
\Edge(v2)(v4)
\end{tikzpicture}
    \end{subfigure}
    \begin{subfigure}{.33\textwidth}
      \centering
      \begin{tikzpicture}
\GraphInit[unit=3,vstyle=Normal]
\SetVertexNormal[Shape=circle, FillColor = black, MinSize=3pt]
\tikzset{VertexStyle/.append style = {inner sep = \inners,outer sep = \outers}}
\SetVertexNoLabel
\Vertex[x = 1.44128, y = 0.97014]{v1}
\Vertex[x = 0.036, y = 0.86946]{v0}
\Vertex[x = 0.72938, y = 2.5378]{v3}
\Vertex[x = 0.7919, y = 0.036]{v2}
\Vertex[x = 0.056102, y = 1.6691]{v5}
\Vertex[x = 1.43066, y = 1.62654]{v4}
\Edge(v0)(v3)
\Edge(v4)(v5)
\Edge(v0)(v2)
\Edge(v0)(v5)
\Edge(v1)(v3)
\Edge(v0)(v4)
\Edge(v3)(v4)
\Edge(v1)(v5)
\Edge(v0)(v1)
\Edge(v2)(v5)
\Edge(v1)(v2)
\Edge(v3)(v5)
\Edge(v2)(v4)
\end{tikzpicture}
    \end{subfigure}
    \par\bigskip
    \begin{subfigure}{.33\textwidth}
      \centering
      \begin{tikzpicture}
\GraphInit[unit=3,vstyle=Normal]
\SetVertexNormal[Shape=circle, FillColor = black, MinSize=3pt]
\tikzset{VertexStyle/.append style = {inner sep = \inners,outer sep = \outers}}
\SetVertexNoLabel
\Vertex[x = 1.46994, y = 0.989]{v1}
\Vertex[x = 0.054778, y = 0.91242]{v0}
\Vertex[x = 0.67678, y = 2.545]{v3}
\Vertex[x = 0.80808, y = 0.036]{v2}
\Vertex[x = 1.40618, y = 1.66]{v5}
\Vertex[x = 0.036, y = 1.58814]{v4}
\Edge(v0)(v3)
\Edge(v4)(v5)
\Edge(v0)(v2)
\Edge(v0)(v5)
\Edge(v1)(v3)
\Edge(v1)(v4)
\Edge(v3)(v4)
\Edge(v0)(v1)
\Edge(v2)(v5)
\Edge(v1)(v2)
\Edge(v3)(v5)
\Edge(v2)(v4)
\end{tikzpicture}
    \end{subfigure}%
    \begin{subfigure}{.33\textwidth}
      \centering
      \begin{tikzpicture}
\GraphInit[unit=3,vstyle=Normal]
\SetVertexNormal[Shape=circle, FillColor = black, MinSize=3pt]
\tikzset{VertexStyle/.append style = {inner sep = \inners,outer sep = \outers}}
\SetVertexNoLabel
\Vertex[x = 1.61832, y = 0.51324]{v1}
\Vertex[x = 0.06159, y = 0.44664]{v0}
\Vertex[x = 0.7835, y = 1.82894]{v3}
\Vertex[x = 0.86098, y = 0.036]{v2}
\Vertex[x = 0.036, y = 1.3396]{v5}
\Vertex[x = 1.57708, y = 1.4073]{v4}
\Edge(v0)(v3)
\Edge(v2)(v3)
\Edge(v0)(v2)
\Edge(v4)(v5)
\Edge(v0)(v5)
\Edge(v1)(v3)
\Edge(v0)(v4)
\Edge(v1)(v4)
\Edge(v3)(v4)
\Edge(v1)(v5)
\Edge(v0)(v1)
\Edge(v2)(v5)
\Edge(v1)(v2)
\Edge(v3)(v5)
\Edge(v2)(v4)
\end{tikzpicture}
    \end{subfigure}
    \caption{The fourteen six-vertex star graphs.\label{fig:six_star}}
\end{figure}
\section{Concluding Remarks}

This work introduced the class of star graphs -- the intersection graphs of the induced maximal stars of some graph.
We presented various results, such as a Krausz-type characterization for the class, a quadratic bound on the size of potential pre-images, membership of the recognition problem in \textsf{NP} and a monotonicity theorem for graphs which are not star-critical.
We also presented a series of properties the members of the class must satisfy, such as being biconnected and that every edge must belong to some triangle.
We leave two main open questions.
The first, and perhaps more challenging of the two, is the complexity of the recognition problem; for example, the complexity of the clique graph recognition problem was left open for many decades, only being settled recently~\cite{clique_recognition} through a series of non-intuitive gadgets and other novel characterizations.
The second is a complete characterization of both star-critical and non-star-critical vertices; in particular, non-star-critical vertices seem the biggest obstacle one must overcome to achieve a linear bound on the size of star-critical pre-images.

Despite our special interests in the above questions, many other directions are available for investigation.
In terms of the class of all star graphs, our best membership checking tool at the moment is generating pre-image candidates and apply Corollary~\ref{cor:frontier} of Theorem~\ref{thm:monotonicity} to prune the search space; if our hypothesis that the recognition problem is \textsf{NP-hard} is indeed true, and thus unlikely solvable in polynomial time, then what is the best way to verify membership?
In a more general context, is there a polynomial delay algorithm that generates all star graphs of a certain order?
We have also only begun the study of the iterated star operator, and various inquiries can be made about its properties, such as convergence/divergence criteria or other structural parameters, like maximum/minimum degree and connectivity.
A strongly related but significantly different open topic is that of edge-star graphs, i.e., the edge-intersection graph of the maximal stars.
Edge-biclique graphs have very recently been studied by Legay and Montero~\cite{edge_biclique} and present significant differences from the vertex-intersection biclique graph of Groshaus et al.~\cite{biclique_graph};
Perhaps the interplay between the edge-star and edge-biclique graphs can yield useful observations for both classes.
\section*{Acknowledgements} We thank the research agencies CAPES, CONICET, CNPq, and FAPEMIG for partially funding this work.

\bibliographystyle{splncs03}
\bibliography{main}
\end{document}